\documentclass[a4paper, cleveref, autoref, thm-restate]{lipics-v2021}

\hideLIPIcs

\usepackage{graphicx}
\newsavebox{\imagebox}
\usepackage{amsthm, amsmath, amssymb}
\usepackage{physics}
\usepackage{mathtools}
\usepackage{bm}
\usepackage{enumerate,comment,cite}
\usepackage{xspace}
\usepackage[usenames,dvipsnames]{xcolor}
\usepackage{MnSymbol}

\graphicspath{{./figures/}}

\newtheorem{question}{Question}

\newcommand{\chiso}{\operatorname{\chi_{\rm so}}}
\newcommand{\chio}{\operatorname{\chi_{\rm o}}}
\newcommand{\tw}{\operatorname{tw}}
\newcommand{\rtw}{\operatorname{rtw}}

\newcommand{\N}{\mathbb{N}}
\DeclarePairedDelimiterX\set[1]\lbrace\rbrace{\def\given{\;\delimsize\vert\;}#1}

\newcommand{\leqsub}[1]{\scalebox{0.8}{$\leq$}#1}

\usepackage{calc}
\usepackage{tikz}
\usetikzlibrary{decorations.pathmorphing, decorations.pathreplacing, decorations, babel, calc, arrows.meta}

\definecolor{IPEgold}{RGB}{255,215, 0}

\nolinenumbers

\title{Strong odd coloring in minor-closed classes}

\author{Miriam Goetze}{Karlsruhe Institute of Technology, Germany}{miriam.goetze@kit.edu}{https://orcid.org/0000-0001-8746-522X}{supported by the Deutsche Forschungsgemeinschaft (DFG, German Research Foundation) -- 520723789.}

\author{Kolja Knauer}{Departament de Matem\`atiques i Inform\`atica, Universitat de Barcelona, Spain}{kolja.knauer@ub.edu}{https://orcid.org/0000-0002-8151-2184}{partially supported by the Severo Ochoa and Mar\'ia de Maeztu Program for Centers and Units of Excellence in R\&D (CEX2020-001084-M) and the  \emph{Ministerio de Econom\'ia, Industria y Competitividad} through grant PID2022-137283NB-C22.}

\author{Fabian Klute}{Universitat Politècnica de Catalunya, Spain}{fabian.klute@upc.edu}{https://orcid.org/0000-0002-7791-3604}{supported by María Zambrano grant 2022UPC-MZC-94041 funded by the Spanish Ministry of Universities and the European Union (NextGenerationEU) and by grant PID2023-150725NB-I00 funded by MICIU/AEI/10.13039/501100011033.}

\author{Irene Parada}{Universitat Politècnica de Catalunya, Spain}{irene.parada@upc.edu}{https://orcid.org/0000-0003-3147-0083}{Serra H\'unter Fellow, partially supported by grant PID2023-150725NB-I00 funded by MICIU/AEI/10.13039/501100011033 and by the UPC grant ALECTORS.}

\author{Juan Pablo Peña}{DIM-CMM, Universidad de Chile, Chile}{juan.pena@dim.uchile.cl}{https://orcid.org/0009-0008-0666-6264}{supported by Basal program FB210005 and Doctoral Fellowship grant 21211955, ANID, Chile.}

\author{Torsten Ueckerdt}{Karlsruhe Institute of Technology, Germany}{torsten.ueckerdt@kit.edu}{https://orcid.org/0000-0002-0645-9715}{supported by the Deutsche Forschungsgemeinschaft (DFG, German Research Foundation) -- 520723789.}

\authorrunning{M. Goetze, K. Knauer, F. Klute, I. Parada, J. P. Peña, T. Ueckerdt}



\begin{document}

\maketitle

\begin{abstract}
    We show that the strong odd chromatic number on any proper minor-closed graph class is bounded by a constant. We almost determine the smallest such  constant for outerplanar graphs.
\end{abstract}

\section{Introduction}
\label{sec:introduction}

Recently, Kwon and Park~\cite{kwon2024strong} introduced the notion of \emph{strong odd colorings} of graphs.
Here, a proper vertex-coloring $\varphi\colon V(G) \to [t]$ of a graph $G$ is \emph{strong odd} if for every vertex $v\in V(G)$ and every color $i \in [t]$ the quantity $|N(v) \cap \varphi^{-1}(i)|$ is either zero or odd.
In other words, among the neighbors of any vertex $v \in V(G)$ every color $i \in [t]$ appears an odd number of times or not at all.
The \emph{strong odd chromatic number} $\chiso(G)$ of $G$ is the minimum $t$ such that $G$ admits a strong odd coloring with $t$ colors.
See \cref{fig:examples_strong_odd} for some examples.

\begin{figure}[ht]
        \savebox{\imagebox}{\includegraphics[page=3]{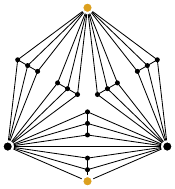}}
        \centering
        \begin{subfigure}[b]{0.25\textwidth}
        \centering\raisebox{\dimexpr.5\ht\imagebox-.5\height}{\includegraphics[page=2]{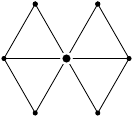}}
        \caption{}
        \label{fig:example_outerplanar}
        \end{subfigure}\hfill
        \begin{subfigure}[b]{0.35\textwidth}
        \centering\usebox{\imagebox}
        \caption{}
        \label{fig:example_planar}
        \end{subfigure}\hfill
        \begin{subfigure}[b]{0.4\textwidth}
        \centering\raisebox{\dimexpr.5\ht\imagebox-.5\height}{\includegraphics[page=4]{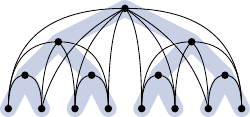}}
        \caption{}
        \label{fig:construction_odd_vs_strong_odd}
        \end{subfigure}
    \caption{
        (\subref{fig:example_outerplanar}) An outerplanar graph~$G'$ with $\chiso(G') = 7$. (\subref{fig:example_planar}) A planar graph~$G''$ with $\chiso(G'') = 20$. (\subref{fig:construction_odd_vs_strong_odd}) The graph $G_3$ from \cref{obs:nobinding} with $\chiso(G_3) = 2^3+1$.
        The black vertices in~$G'$ and~$G''$ have pairwise distinct colors in every strong odd coloring.
    }
    \label{fig:examples_strong_odd}
\end{figure}

\subparagraph*{Related Colorings.}

Strong odd colorings are closely related to a number of similar concepts.
A proper vertex-coloring of $G$ is \emph{odd} if in the neighborhood of every non-isolated vertex some color appears an odd number of times.
The corresponding parameter $\chio(G)$ is called the \emph{odd chromatic number}.
A proper vertex-coloring of $G$ is \emph{conflict-free} if in the neighborhood of every non-isolated vertex some color appears exactly once.
The corresponding parameter is denoted by $\chi_{\mathrm{pcf}}(G)$.
Both, $\chio$ and $\chi_{\mathrm{pcf}}$ have received plenty of attention in recent years, see e.g.~\cite{Hic23,Liu24, PS22,jklmpqswyz23,zbMATH07615629,zbMATH07622664,zbMATH07756975,zbMATH07785865,zbMATH07940890,zbMATH07959405,zbMATH07962118,zbMATH07975074}.

If we require that among the neighbors of every vertex, \emph{every} color appears exactly once or not at all, this corresponds to coloring the \emph{square} $G^2$ of $G$.
Here, $G^2$ has the same vertex set as $G$ with two vertices being adjacent in $G^2$ if their distance in $G$ is either $1$ or $2$.
The chromatic number of $G^2$ is studied in particular for planar graphs with respect to Wegner's Conjecture~\cite{Weg77}, see also~\cite{Cra23}.
From the above definitions, we immediately get the following inequalities.

\begin{equation}
\begin{minipage}{10cm}
    \begin{tikzpicture}
	\coordinate (dist_hor) at (2.2,0);
	\coordinate (dist_vert) at (0,0.5);
	\node (a) at (0,0) {$\chi(G)$};
        \node (b) at (2,0) {$\chio(G)$};
	\node (c) at ($(2,0)+(dist_hor)+(dist_vert)$) {$\chiso(G)$};
	\node (d) at ($(2,0)+(dist_hor)-(dist_vert)$) {$\chi_{\mathrm{pcf}}(G)$};
        \node (e) at ($(2,0)+2*(dist_hor)$) {$\chi(G^2)$};
        \node (f) at (8.6,0) {$\Delta(G)^2+1$};

        \path[color=white] (a) edge node[black,pos=0.5] {$\leq$} (b);
	\path[color=white] (b) edge node[black,pos=0.5,sloped] {$\leq$} (c);
	\path[color=white] (b) edge node[black,pos=0.5,sloped] {$\leq$} (d);
	\path[color=white] (c) edge node[black,pos=0.5,sloped] {$\leq$} (e);
	\path[color=white] (d) edge node[black,pos=0.5,sloped] {$\leq$} (e);
        \path[color=white] (e) edge node[black,pos=0.5] {$\leq$} (f);
    \end{tikzpicture}
\end{minipage}
\label{eq:inequalities}
\end{equation}

Here, $\Delta(G)$ denotes the maximum degree of $G$.
The rightmost inequality in \eqref{eq:inequalities} can also be reversed in some sense, since trivially $\chi(G^2) > \Delta(G)$ for all graphs $G$.
However, all remaining inequalities in~\eqref{eq:inequalities} can be arbitrarily far apart.
Petru\v{s}evski and \v{S}krekovski~\cite{PS22} show that there is no function $f$ such that $\chio(G)\leq f(\chi(G))$ for all $G$.
It is also known~\cite{jklmpqswyz23} that there is no function $f$ such that $\chi_{\mathrm{pcf}}(G)\leq f(\chio(G))$ for all $G$.
Further, since $\Delta(G) < \chi(G^2)$, any graph class with unbounded degree but bounded $\chi_{\mathrm{pcf}}$ shows that there is no function $f$ such that $\chi(G^2)\leq f(\chi_{\mathrm{pcf}}(G))$ for all $G$, e.g. bipartite permutation graphs, see~\cite{jklmpqswyz23}.
Similarly, Kwon and Park~\cite{kwon2024strong} note that there is no function $f$ such that $\chi(G^2)\leq f(\chiso(G))$ for all $G$, while they show graphs $G$ with $\chiso(G)\in \Omega(\chio(G)^2)$.
Here, we improve the latter.

\begin{observation}\label{obs:nobinding}
    There is no function $f$ such that $\chiso(G)\leq f(\chio(G))$ for all graphs~$G$.
\end{observation}
\begin{proof}
    For every integer $k\geq 1$ let $G_k$ be the bipartite graph whose vertex set consists of the vertices of the full rooted binary tree $T_k$ of height $k$, i.e., the distance of the root to every leaf is $k$.
    For every leaf $v$ of $T_k$, in $G_k$ we put all edges between $v$ and its ancestors, see \cref{fig:construction_odd_vs_strong_odd} for an example.
    It is known that $\chio(G_k)\leq 4$ for all $k$~\cite{jklmpqswyz23}.
    
    We show by induction on $k$, that in order to color the leaves such that all colors appearing in the neighborhood of any internal vertex appear an odd number of times, at least $2^k$ colors are needed.
    This clearly holds for $k=1$.
    For $k \geq 2$, take such a coloring $\varphi$ of the leaves of $G_k$.
    Note that $G_k$ consists of two disjoint copies $G', G''$ of $G_{k-1}$ plus the root vertex $r$ adjacent to all leaves in $G'$ and $G''$.
    Further, restricting $\varphi$ to the leaves of $G'$ and $G''$ gives feasible colorings of $G'$ and $G''$.
    Hence, each color appearing on the leaves of $G'$ appears an odd number of times on the leaves of $G'$ and similarly for $G''$.
    Hence, no color can be used on the leaves of $G'$ and on the leaves of $G''$.
    Thus, we need twice as many colors as for $G_{k-1}$, i.e., by induction $\varphi$ uses at least $2\cdot 2^{k-1}=2^k$ colors on the leaves.
    
    Finally, as the root of $T_k$ is connected to all leaves, we get that $\chiso(G_k) \geq 2^k+1$.
\end{proof}

\subparagraph*{Planar and Outerplanar Graphs.}

Kwon and Park~\cite{kwon2024strong} ask whether the strong odd chromatic number is bounded on the class $\mathcal{P}$ of all planar graphs.
This is answered affirmatively by Caro, Petru\v{s}evski, \v{S}krekovski, and Tuza~\cite{caro2024strong} who show that for
\[
    c_{\mathcal{P}} = \max\{ \chiso(G) \mid G \text{ planar}\} \qquad \text{we have} \qquad 12 \leq c_{\mathcal{P}} \leq 388.
\]
We increase the lower bound to $20$, which gives a negative answer to \cite[Problem 3.6.1]{caro2024strong}.

\begin{observation}\label{planar}
    For the planar graph~$G$ in \cref{fig:example_planar} we have~$\chiso(G) = 20$.
\end{observation}

We leave the analysis of the example to the reader.
For the class $\mathcal{O}$ of all outerplanar graphs, the authors~\cite{caro2024strong} show that for
\[
    c_{\mathcal{O}} = \max\{ \chiso(G) \mid G \text{ outerplanar}\} \qquad \text{we have} \qquad 7 \leq c_{\mathcal{O}} \leq 30.
\]
The lower bound of~$7$ is given by the outerplanar graph in \cref{fig:example_outerplanar}, and suspected to be the correct value~\cite[Problem 3.6.2]{caro2024strong}.
We reduce the upper bound to $8$.

\begin{proposition}\label{prop:outerplanar-at-most-8}
    For every outerplanar graph $G$ we have $\chiso(G) \leq 8$.
\end{proposition}

\subparagraph*{Proper Minor-Closed Classes.}

The class $\mathcal{P}$ of all planar graphs and the class $\mathcal{O}$ of all outerplanar graphs are closed under taking minors.
Our main result is that the strong odd chromatic number is bounded on every proper\footnote{different from the class of all graphs} minor-closed graph class.
Hence, this is a far-reaching generalization of the results above, concerning only the classes $\mathcal{P}$ and $\mathcal{O}$.

\begin{theorem}\label{thm:minor-closed}
    For every proper minor-closed graph class~$\mathcal{G}$, there exists a constant~$c_{\mathcal{G}}$ such that for every graph~$G \in \mathcal{G}$ we have~$\chiso(G) \leq c_{\mathcal{G}}$.
\end{theorem}

Our proof of \cref{thm:minor-closed} goes in three steps, which are briefly summarized as follows.

\begin{description}
    \item[(1) Bounded Treewidth.]\label{step-1}
        First, we look at graphs of bounded treewidth and prove that there is a function $f_1$ such that $\chiso(G) \leq f_1(\tw(G))$ for all $G$.
        Given any graph $G$ with $\tw(G) = k$, we consider any $k$-tree\footnote{edge-maximal graph of treewidth~$k$} $H$ with $G \subseteq H$ and use a BFS-layering $L_0,L_1,\ldots$ of $H$, where $L_d = \{v \in V(H) \mid {\rm dist}_H(v,r) = d\}$ is the set of vertices at distance $d$ to a fixed root vertex $r$.
        Crucially, for each $d \geq 0$ we get induced subgraphs $G[L_d] \subseteq H[L_d]$ of treewidth at most~$k-1$, which allows us to do induction on $k$.
        
        The hardest part is to account for edges between layers, i.e., to ensure that for every vertex $v \in L_d$ and every color $i$ used on $L_{d+1}$ the quantity $|N(v) \cap \varphi^{-1}(i)|$ is zero or odd.
        We do this with a stronger claim involving some additional subsets $M_1,\ldots,M_\ell \subseteq L_{d+1}$ that must intersect every color class in an odd number of elements or not at all.
        Then, reusing colors on every third layer, we construct a proper $f_1(k)$-coloring~$\varphi$ of~$H$ that is strong odd for $G$.
        Our function $f_1(k)$ is exponential in $k$, which is however unavoidable, as for example, $G_k$ from \cref{obs:nobinding} has $\tw(G_k) = k$ and \mbox{$\chiso(G_k) = 2^k+1$}.
        (The color of the root of $T_k$ does not appear on any leaf of $T_k$.)
        
        We remark that outerplanar graphs have treewidth at most~$2$.
        And indeed our proof of \cref{prop:outerplanar-at-most-8} (outerplanar graphs have strong odd $8$-colorings) also works along a BFS-layering of a $2$-tree $H$ with $G \subseteq H$, but let us refer to \cref{sec:outerplanar} for the details.

    \item[(2) Bounded Row-Treewidth.]
        The next step are graphs of bounded row-treewidth, that is graphs $G$ with $G \subseteq H \boxtimes P$ where $H$ has small treewidth and $P$ is a path.
        The \emph{strong product} $G' \boxtimes G''$ of two graphs is the graph with vertex set $V(G') \times V(G'')$ having an edge between two distinct vertices $(u,p)$ and $(v,q)$ if and only if ($uv \in E(G')$ or $u=v$) and ($pq \in E(G'')$ or $p=q$).
        
        The \emph{row-treewidth} $\rtw(G) = \min\{ k \mid G \subseteq H \boxtimes P, \tw(H) = k\}$ has gained a lot of prominence around the so-called Graph Product Stucture~\cite{DHJLW21}.
        Using the layer-structure of $H \boxtimes P$ together with the bound $\chiso(H) \leq f_1(\tw(H))$ above, we prove that $\chiso(G) \leq f_2(\rtw(G))$ for a universal (again exponential) function $f_2$ and all $G$.

    \item[(3) Clique Sums.]
        Third, we extend our results to subgraphs of so-called $(w,k,t)$-sums.
        A \emph{$\leqsub{w}$-clique-sum} is obtained from two graphs $G',G''$ by identifying a clique $C'$ of size at most $w$ in $G'$ with a clique $C''$ of the same size in $G''$.
        For integers $w,k,t \geq 1$, a \emph{$(w,k,t)$-sum} is a graph $G$ that can be obtained by taking $\leqsub{w}$-clique-sums of graphs $G_1,\ldots,G_n$ where each $G_i$ is obtained from $H_i \boxtimes P$ for some $k$-tree $H_i$ after adding $t$ universal vertices.
        A result of Dujmović, Joret, Micek, Morin, Ueckerdt, and Wood~\cite[Theorem 41]{DJMMUW20} states that for every proper minor-closed graph class $\mathcal{G}$ there are constants $k$ and $t$, such that every $G \in \mathcal{G}$ is a subgraph of a $(2(k+1)+t,k,t)$-sum.

        The iterative clique-sums give rise to yet another layering structure (similar to BFS-layerings).
        Using this, we give a universal function $f_3$ such that $\chiso(G) \leq f_3(w,k,t)$ whenever $G$ is a subgraph of a $(w,k,t)$-sum, which concludes the proof of \cref{thm:minor-closed}.
\end{description}

\subparagraph*{Organization of the paper.}

In \cref{sec:preliminaries} we explain the concepts used throughout the paper.
This includes a generalization of strong odd coloring to directed graphs, but also a discussion of the required properties of graphs of small treewidth, or small row-treewidth, and how these relate to proper minor-closed classes of graphs.
In particular, we define BFS-layerings and natural layerings in \cref{sec:preliminaries}.

\cref{sec:outerplanar} is devoted to the proof of \cref{prop:outerplanar-at-most-8}, while \cref{sec:treewidth,sec:proper_minor_closed,sec:row-treewidth} individually address the three steps to prove \cref{thm:minor-closed} as described above.
That is, we show that the strong odd chromatic number is bounded for graphs of small treewidth in \cref{sec:treewidth}, for graphs of small row-treewidth in \cref{sec:row-treewidth}, and finally for graphs from a proper minor-closed class in \cref{sec:proper_minor_closed}.

After briefly mentioning some application of strong odd colorings to facially odd colorings in \cref{sec:app}, we conclude with a discussion of possible extensions and open questions in \cref{sec:further}.

\section{Preliminaries}
\label{sec:preliminaries}

For an integer~$k \in \N$, we write~$[k]$ for the set~$\set{1, \dots,k}$ of the first~$k$ integers.
For a graph~$G$ and a set~$S \subseteq V(G)$, we denote by~$G[S]$ the subgraph of~$G$ induced by the vertices of~$S$. For convenience we restate:

\begin{definition}\label{def:strong-odd-coloring}
    For a graph $G$, a $t$-coloring $\psi \colon V \to [t]$ of $V$ is a \emph{strong odd coloring} if 
    \begin{enumerate}
     \item $\psi$ is \emph{proper}, i.e., for any $uv \in E(G)$ we have $\psi(u) \neq \psi(v)$, and
 
     \item $\psi$ is \emph{strong odd}, i.e., for every $v \in V(G)$ and every color $c \in [t]$ the cardinality of\\
     $N(v) \cap \psi^{-1}(c) = \{u \in V(G) \mid uv \in E(G), \psi(u) = c\}$ is either zero or odd.
    \end{enumerate}
\end{definition} 
For a graph $G$, its \emph{strong odd chromatic number} $\chiso(G)$ is the minimum $t$ such that $G$ admits a strong odd coloring with $t$ colors.

\subparagraph*{Strong Odd Colorings of Directed Graphs.}

We say that a vertex-coloring~$\psi$ of a graph~$G$ is \emph{strong odd} on a set~$M \subseteq V(G)$ if for every color~$c$ the cardinality of~$\psi^{-1}(c) \cap M$ is either odd or zero.

We consider directed graphs which may contain edges in both directions, but no multi-edges or loops.
A directed graph~$\vec{G}$ consists of a vertex-set $V(\vec{G})$ and a set $E(\vec{G})$ of directed edges.
We denote an edge in $E(\vec{G})$ directed from vertex~$u$ to vertex~$v$ by~$\vec{uv}$.
In particular, $\vec{uv} \neq \vec{vu}$ here, while we allow that none, one, or even both of $\vec{uv},\vec{vu}$ are present in the edge set~$E(\vec{G})$ of a directed graph~$\vec{G}$.
A directed graph~$\vec{H}$ is a subgraph of a directed graph~$\vec{G}$ if $V(\vec{H}) \subseteq V(\vec{G})$ and $E(\vec{H}) \subseteq E(\vec{G})$.
For a directed graph $\vec{G}$ and a vertex $v \in V(\vec{G})$, we denote by 
\[
    N^+(\vec{G},v) = \{u \in V(\vec{G}) \mid \vec{vu} \in E(\vec{G})\}
\] 
the set of all \emph{out-neighbors} of $v$.

Now a $t$-coloring $\psi \colon V(\vec{G}) \to [t]$ of the vertices of $\vec{G}$ is a \emph{strong odd coloring} if
\begin{enumerate}
    \item $\psi$ is \emph{proper}, i.e., for any $\vec{uv} \in E(\vec{G})$ we have $\psi(u) \neq \psi(v)$, and
    \item $\psi$ is strong odd on the out-neighborhood~$N^+(\vec{G},v)$ of every vertex~$v \in V(\vec{G})$.
\end{enumerate}
That is, for each vertex $v$ each color $c$ appears among the out-neighbors of $v$ either not at all, or an odd number of times.
In particular, if for every $\vec{uv} \in E(\vec{G})$ we also have $\vec{vu} \in E(\vec{G})$, then we can think of $\vec{G}$ as being an undirected graph and the definition of a strong odd $t$-coloring of $\vec{G}$ coincides with \cref{def:strong-odd-coloring}.

\subparagraph*{\boldmath $k$-Trees and BFS-Layerings.}

For an integer $k \geq 0$, a \emph{$k$-tree} is a graph that can be recursively defined as follows: 
a $k$-clique\footnote{While it is more common to start with a $(k+1)$-clique, it is more convenient for us (and equivalent for our purposes) to start with a $k$-clique.} is a $k$-tree and a graph obtained from a $k$-tree by adding a vertex and connecting it to all the vertices of an existing $k$-clique is a $k$-tree. 
A subgraph of a $k$-tree is a \emph{partial $k$-tree}. 
It is well known that partial $k$-trees are exactly the graphs of treewidth at most $k$.

Let~$Q$ be the initial $k$-clique in the construction sequence of a $k$-tree~$G$. 
Let~$G'$ be the graph obtained from~$G$ by adding a new vertex~$r$ that is connected to every vertex of~$Q$.
Consider the partition of the vertices of~$G$ into layers~$L_1, \dots, L_t$ where the $i$-th layer contains all vertices with distance~$i$ to the vertex~$r$.
We call such a partition a \emph{BFS-layering}, see \cref{fig:example-BFS-layering} for an example.
We remark that Dujmovi\'{c}, Morin and Wood~\cite[Theorem 6.1]{DujmovicMW05} considered a similar concept for a generalization of~$k$-trees.

\begin{figure}
    \centering
    \includegraphics[page=1]{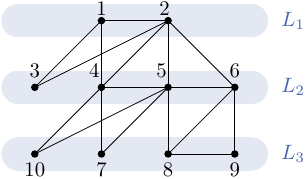}
    \caption{A BFS-layering~$L_1,L_2,L_3$ of a $2$-tree~$G$. The labels of the vertices correspond to their insertion order in the construction sequence of~$G$.}
    \label{fig:example-BFS-layering}
\end{figure}

\begin{observation}
    \label{obs:bfs-layering}
    Let $k \geq 1$ be an integer. 
    If $L_1, \dots, L_d$ is a BFS-layering of a $k$-tree~$G$ then the following holds:
    \begin{enumerate}[(B1)]
        \item\label{property:root_k-clique}
            The first layer~$L_1$ is a $k$-clique of~$G$.
        \item \label{property:parent_k-clique} 
            If~$L_i$ is a layer with $i\geq 2$ and~$C_i$ a component of the induced subgraph~$G[L_i]$, then the set of vertices in~$L_{i-1}$ with a  neighbor in~$C_i$ forms a $k$-clique of~$G$.
        \item \label{property:layer_tw_k-1} The induced subgraph~$G[L_i]$ of each layer~$L_i$ is a partial $(k-1)$-tree.
        \item\label{property:is_layering} For every~$i \in [\ell]$, vertices in~$L_i$ only have neighbors in~$L_{i-1},L_i$ and~$L_{i+1}$.
    \end{enumerate}
\end{observation}

\subparagraph*{Row-Treewidth.}

The \emph{strong product} of two graphs~$H$ and~$H'$, denoted by~$H \boxtimes H'$, is the graph on the vertex-set~$V(H) \times V(H')$ where two vertices~$(u,x)$ and~$(v,y)$ are adjacent if and only if one of the following holds
\begin{itemize}
    \item $u=v$ and $xy \in E(H')$
    \item $x=y$ and $uv \in E(H)$
    \item $uv \in E(H)$ and $xy \in E(H')$.
\end{itemize}
Suppose the graph~$H'$ corresponds to a path~$P$ on~$n$ vertices and let~$x_1, \dots, x_n$ be the order of the vertices along~$P$.
We then obtain a partition of the vertex-set of~$H \boxtimes P$ into $n$~\emph{layers} $L_1, \dots, L_n$ where the $i$-th layer~$L_i = \set{(u,x_i) \given u \in V(H)}$ corresponds to the copies of~$x_i$ in the $H \boxtimes P$.
Note that a vertex of the $i$-th layer~$L_i$ is only adjacent to vertices of up to three layers, namely~$L_{i-1}, L_i$ and~$L_{i+1}$.
Further, every layer induces a copy of~$H$ in the graph~$H \boxtimes P$.

We say that a graph~$G$ has \emph{row-treewidth} at most~$k$, denoted by~$\rtw(G) \leq k$, if there exists a $k$-tree~$H$ and a path~$P$ such that~$G$ is a subgraph of the strong product of~$H$ and~$P$, i.e., $G \subseteq H \boxtimes P$.

\subparagraph*{Proper Minor-Closed Graph Classes and Natural Layerings.}

For graphs~$G$ and~$H$, we denote by~$G+H$ their \emph{join}, i.e., the graph that is obtained from the disjoint union of~$G$ and~$H$ by connecting all pairs $(u,v) \in V(G) \times V(H)$ of vertices with an edge, i.e. $V(G+H)=V(G) \cup V(G)$ and~$E(G+H) = E(G) \cup E(H) \cup \set{uv \given u \in V(G), v \in V(H)}$.

We say that a graph~$F$ is a \emph{$(k,t)$-summand} if there exists a $k$-tree~$H$ and a path~$P$ such that~$F = (H \boxtimes P) + K_t$.
A $(w,k,t)$-sum is a graph recursively defined as follows: 
a $(k,t)$-summand is a $(w,k,t)$-sum. 
If~$S$ is a $(w,k,t)$-sum with a clique~$Q_S$ of size at most~$w$ and~$F$ a $(k,t)$-summand with a clique~$Q_F$ of size~$\abs{Q_S}$, then the graph obtained by identifying the vertices of~$Q_S$ and~$Q_F$ is a $(w,k,t)$-sum.
Note in particular that $(w,k,t)$-sums are not necessarily connected, as the common clique may have size~$0$.
The largest clique in a $(k,t)$-summand~$F$ has size at most~$2(k+1)+t$, as a clique in~$F$ contains vertices of at most two different layers of~$H \boxtimes P$ and all vertices of~$K_t$. 
Thus, the largest clique in a $(w,k,t)$-sum~$S$ also has size at most~$2(k+1)+t$ as each clique of~$S$ is completely contained in one of its summands.

\begin{theorem}[{\cite[Theorem 41]{DJMMUW20}}]
    \label{thm:minor-closed-sums}
    For every proper minor-closed graph class~$\mathcal{G}$, there are integers~$k$ and~$t$ such that every graph~$G \in \mathcal{G}$ is a subgraph of a $(2(k+1)+t,k,t)$-sum.
\end{theorem}

We will show in \cref{sec:proper_minor_closed} that the strong odd chromatic number~$\chiso$ of subgraphs of $(w,k,t)$-sums is bounded in~$k$ and~$t$ (cf. \cref{prop:sum-fine}).
Together with \cref{thm:minor-closed-sums} this implies that the strong odd chromatic number is bounded for every proper minor-closed graph class, see~\cref{thm:minor-closed}.
The proof of the result for $(w,k,t)$-sums is centered around natural layerings, a structure similar to BFS-layerings in the case of $k$-trees.

We say that a partition~$L_1, \dots, L_{\ell}$ of the vertices of a $(w,k,t)$-sum~$S$ with~$w \geq 1$ is a \emph{natural layering of~$S$} if the following holds:
\begin{enumerate}[(N1)]
    \item\label{property_layering_sum:root-layer} The first layer~$L_1$ is a $(0,k,t)$-sum.
    \item\label{property_layering_sum:parent-clique} If~$L_i$ is a layer with~$i \geq 2$ and~$C_i$ a component of the induced subgraph~$S[L_i]$, then the set of vertices in~$L_{i-1}$ with a neighbor in~$C_i$ forms a clique of size at most~$w$.
    \item\label{property_layering_sum:layer_w-1} The induced subgraph~$S[L_i]$ of each layer~$L_i$ is a subgraph of a $(w-1,k,t)$-sum.
    \item\label{property_layering_sum:is_layering} For every~$i \in [\ell]$, vertices in~$L_i$ only have neighbors in~$L_{i-1},L_i$ and~$L_{i+1}$.
\end{enumerate}
See \cref{fig:example-clique_sum_layering} for an illustration.

\begin{lemma}
    \label{lem:sum_natural_layering}
    Let~$k,t \geq 0$ and $w \geq 1$ be integers.
    Every $(w,k,t)$-sum admits a natural layering.
\end{lemma}
\begin{proof}
    We proceed by induction on the number~$s$ of~$(k,t)$-summands.
    If~$s = 1$, all vertices of~$S$ are assigned to the first layer~$L_1$.
    This layering clearly satisfies (N\ref{property_layering_sum:root-layer})-(N\ref{property_layering_sum:is_layering}).
    
    Now let~$s \geq 2$ and let~$F_s$ be the last summand that is attached to a clique~$Q$.
    Consider the $(w,k,t)$-sum~$S'$ obtained from~$S$ by deleting all vertices of~$F_s - Q$.
    By induction, $S'$ admits a natural layering~$L_1', \dots, L_{\ell'}'$. 
    Let~$q$ be the smallest index of a layer that contains vertices of~$Q$.
    We set~$q$ to~$0$ if there is no such layer.
    Assigning the vertices of~$F_s - Q$ to the $(q+1)$-th layer, we obtain a layering~$L_1, \dots, L_{\ell}$ of~$S$, see \cref{fig:example-clique_sum_layering}.
    \begin{figure}
        \savebox{\imagebox}{\includegraphics[page=3]{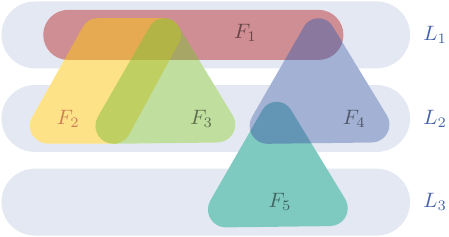}}
        \centering
        \begin{subfigure}[b]{0.4\textwidth}
        \centering\raisebox{\dimexpr.5\ht\imagebox-.5\height}{\includegraphics[page=2]{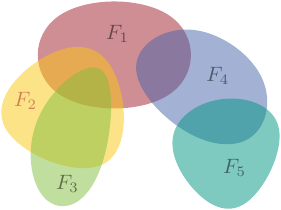}}
        \caption{}
        \label{fig:clique-sum}
        \end{subfigure}\hfill
        \begin{subfigure}[b]{0.6\textwidth}
        \centering\usebox{\imagebox}
        \caption{}
        \label{fig:clique-sum-layering}
        \end{subfigure}
        \caption{(\subref{fig:clique-sum}) A $(w,k,t)$-sum~$S$. Each summand~$F_i$ has a different color. The indices correspond to the insertion order of the summands~$F_i$, their intersections to the cliques on which they are attached. (\subref{fig:clique-sum-layering}) The natural layering of~$S$ obtained by \cref{lem:sum_natural_layering}.}
        \label{fig:example-clique_sum_layering}
    \end{figure}
    Note that we might have created a new layer.

    It remains to show that~$L_1, \dots, L_{\ell}$ is a natural layering of~$S$.
    Property~(N\ref{property_layering_sum:parent-clique}) clearly holds by construction of the layering.
    
    Case 1. $q=0$. 
    The graph~$F_s$ is disjoint from all other summands of~$S$.
    All vertices of~$F_s$ are assigned to the first layer and (N\ref{property_layering_sum:root-layer})-(N\ref{property_layering_sum:is_layering}) hold.
    
    Case 2. $q \geq 1$. 
    The graph~$F_s$ is attached to a clique~$Q$ and at least one vertex of~$Q$ belongs to a layer different from~$L_{q+1}$.
    Property (N\ref{property_layering_sum:root-layer}) holds as the first layer remains unchanged. 
    
    We first observe that~(N\ref{property_layering_sum:is_layering}) still holds. 
    This follows from the fact that every clique of~$S'$ is contained in at most two adjacent layers.
    Thus, the clique~$Q$ only has vertices in~$L_q'$ and~$L_{q+1}'$ and vertices of~$F_s-Q$ only have neighbors in these two layers.
    
    Note that the graph~$S[L_{q+1}]$ is obtained from the induced graph~$S'[L_{q+1}']$ and~$F_s - (Q \cap L_q')$ where we glued along the clique~$Q \cap L_{q+1}'$.
    Hence, (N\ref{property_layering_sum:layer_w-1}) holds. 
\end{proof}

\section{Strong Odd Colorings for Outerplanar Graphs}
\label{sec:outerplanar}

In our proof of~\cref{prop:outerplanar-at-most-8} we make use of the properties of BFS-layerings~$k$-trees similar to \cref{obs:bfs-layering} but adapted to the special case of outerplanar graphs.
Outerplanar graphs have treewidth at most~$2$ and in fact edge-maximal outerplanar graphs are so-called \emph{simple~$2$-trees}, corresponding to the notion of \emph{simple treewidth}~\cite{knauer2012simple}.
Analogous to \cref{obs:bfs-layering}, a BFS-layering $L_1,L_2,\ldots$ of an outerplanar $2$-tree $H$ (i.e., edge-maximal outerplanar graph) has the following properties.
\begin{itemize}
    \item The first layer $L_1$ is an outer edge of $H$.
    \item Each layer $L_i$ with $i \geq 2$ is a vertex-disjoint union of paths, one for each edge in $L_{i-1}$.
    \item For any edge $xy$ on $L_{i-1}$ the corresponding path $P$ on $L_i$ consists of two edge-disjoint subpaths $P_x,P_y$ on vertices $V(P) \cap N_H(x)$ and $V(P) \cap N_H(y)$, respectively.
    The shared vertex of $P_x$ and $P_y$ is the only vertex on $P$ with two neighbors in $L_{i-1}$.
\end{itemize}

\begin{proof}[Proof of \cref{prop:outerplanar-at-most-8}]
    Given an outerplanar graph $G$, we take an edge-maximal outerplanar graph $H$ with $G \subseteq H$ and consider a BFS-layering $L_1,L_2,\ldots$ of $H$ with the above-mentioned properties.
    Let us call a coloring \emph{good on a path $P$} if every two vertices at distance at most~$2$ on $P$ have distinct colors.
    We shall construct an $8$-coloring $\psi \colon V(H) \to [8]$ such that
    \begin{enumerate}[(O1)]
        \item $\psi$ is a proper coloring of $H$,\label{item:proper-H}
        \item for every vertex $v$ of $H$ on any layer $L_i$, any two vertices in $N_H(v) - L_{i+1}$ have distinct colors (in particular $\psi$ is good on every path in every layer), and\label{item:good-on-paths}
        \item $\psi$ restricts to a strong odd coloring on the graph~$G$.\label{item:strong-odd-G}
    \end{enumerate}
    We shall define the coloring $\psi$ layer by layer, starting with colors $1$ and $2$ for the two vertices on $L_1$.
    In fact, we may assume that $L_1$ and $L_2$ contain no vertices of $G$.
    Hence, it is enough to color the path on $L_2$ by repeating the colors $3,4,5,3,4,5,\ldots$ along the path.
    
    To color the layers $L_i$ with $i \geq 3$, we use the following claim.

    \begin{figure}
        \centering
        \includegraphics[page=2]{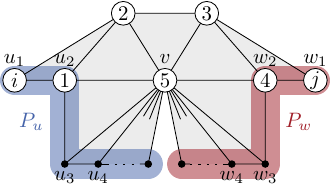}
        \caption{The partially precolored outerplanar graph $A$ in \cref{claim:outerplanar}.}
        \label{fig:outerplanar-claim}
    \end{figure}

    \begin{claim}\label{claim:outerplanar}
        Let $A$ be an outerplanar graph with a partial coloring $\psi \colon V(A) \to [8]$ as depicted in \cref{fig:outerplanar-claim}, consisting of a triangle colored $2,3,5$ and two paths $P_u,P_w$.
        Assume that $\psi$ is proper (i.e., $i \neq 1,2$ and $j \neq 3,4$) and good on the path $[u_1,u_2,v,w_2,w_1]$ (i.e., $i,j \neq 5$).
        
        Let $G_A \subseteq A$ be any subgraph of $A$.
        Then $\psi$ can be extended to a proper $8$-coloring of $A$ that is good on $P_u$ and $P_w$, and strong odd for $G_A$, and satisfies $\psi(u_3) \neq 2$ and $\psi(w_3) \neq 3$.
    \end{claim}
    \begin{claimproof}
        We may assume (by renaming colors if needed) that $i \neq 6$ and $j \neq 8$.
        Let $P'_u = P_u - \{u_1,u_2\}$ and $P'_w = P_w - \{w_1,w_2\}$ be the two paths of uncolored vertices.
        As a preliminary coloring, let $\phi \colon P'_u \cup P'_w \to \{6,7,8\}$ be the (unique) proper coloring that is good on $P'_u$ and $P'_w$ with $\phi(u_3) = 6$ and $\phi(u_4)= \phi(w_4) = 7$ and $\phi(w_3) = 8$.
        Extending $\psi$ by $\phi$ yields a proper coloring of $A$ that is good on $P_u$ and $P_w$.
        If this extension is strong odd for $G_A$, we are done. 
        Otherwise, vertex $v$ has in $G_A$ an even non-zero number of neighbors colored $6$ (or $7$, or $8$).

        If $|N_{G_A}(v) \cap \phi^{-1}(6)|$ is even and non-zero, we consider a color $a \in \{3,4\}$ with $a \neq i$.
        If $v$ is adjacent in $G_A$ to a precolored vertex of color $a$, we give all vertices in $\phi^{-1}(6)$ the color $a$.
        Note that this is still good on $P_u$ and $P_w$ since $w_3,w_4 \notin \phi^{-1}(6)$ and $a \in \{3,4\}$.
        Moreover, $v$ has (in $G_A$) no neighbor of color~$6$ and an odd number of neighbors of color $a$ in the resulting coloring.
        If $v$ is not adjacent in $G_A$ to a precolored vertex of color $a$, we split $\phi^{-1}(6)$ into a set colored $a$ and a set colored $6$, such that $v$ has (in $G_A$) an odd number of neighbors of either color.
        Note again that the resulting coloring is still good on $P_u$ and $P_w$.

        The case that $|N_{G_A}(v) \cap \phi^{-1}(8)|$ is even and non-zero, is symmetric.
        We consider a color $b \in \{1,2\}$ with $b \neq j$ and proceed as in the previous case.
        Since $a \in \{3,4\}$ and $b \in \{1,2\}$, we have $a \neq b$.
        This ensures that the coloring is still good on $P_u$ and $P_w$.

        Finally, if $|N_{G_A}(v) \cap \phi^{-1}(7)|$ is even and non-zero, we consider a color $c \in \{3,4\}$ with $c \neq a$, and a color $d \in \{1,2\}$ with $d \neq b$.
        We split $\phi^{-1}(7)$ into a set containing $u_4$ colored $c$ and a set containing $w_4$ colored $d$, such that $v$ has (in $G_A$) an odd number of neighbors of either color.
        The resulting coloring is again good on $P_u$ since $u_4$ is colored $3$ or $4$ but not $a$, and also good on $P_w$ since $w_4$ is colored $1$ or $2$ but not $b$.

        This way, we ensure that for each color $\alpha \in [8]$, we have that $|N_{G_A}(v) \cap \psi^{-1}(\alpha)|$ is either zero or odd.
        Finally, note that $\psi(u_3) \in \{3,4,6\}$ and hence $\psi(u_3) \neq 2$.
        Symmetrically, $\psi(w_3) \in \{1,2,8\}$ and hence $\psi(w_3) \neq 3$.
        Thus, $\psi$ is a strong odd coloring for the entire $G_A$.
    \end{claimproof}

    Now, let us use \cref{claim:outerplanar} to define the coloring $\psi$ on layer $L_{i+1}$ with $i\geq 2$, given that layers $L_1,\ldots,L_i$ are already colored satisfying (O\ref{item:proper-H})--(O\ref{item:strong-odd-G}).
    Consider one path $P$ on layer $L_i$ and all the neighbors of $P$ in $L_{i+1}$.
    (The neighbors of the other paths in $L_i$ can be treated independently.)
    Let $v$ be the unique vertex on $P$ with two neighbors $x,y$ on $L_{i-1}$.
    Consider the subpath $[u_1,u_2,v,w_2,w_1]$ of $P$ of length five with middle vertex $v$.
    (By choosing $H$, we may assume that $P$ is long enough, i.e., $u_1,u_2,w_2,w_1$ exist.)
    The already colored vertices $u_1,u_2,v,w_2,w_1,x,y$ together with all uncolored neighbors of $v$ on $L_{i+1}$ form a partially precolored outerplanar graph $A$ as in \cref{claim:outerplanar}.
    (Again, we may assume that $u_3,u_4,w_4,w_3$ exist.)
    As the precoloring satisfies (O\ref{item:proper-H}) and (O\ref{item:good-on-paths}), we can rename the colors to match the assumptions of \cref{claim:outerplanar}.
    Thus, the $8$-coloring $\psi$ can be extended to the neighbors of $v$ such that for every $c \in [8]$ the cardinality of $N_G(v) \cap \psi^{-1}(c)$ is either zero or odd.
    Moreover, \cref{claim:outerplanar} gives $\psi(u_3) \neq \psi(x)$ and $\psi(w_3) \neq \psi(y)$.
    In particular, for every so-far colored vertex, any two of its colored neighbors have distinct colors.

    \begin{figure}
        \centering
        \includegraphics[page=3]{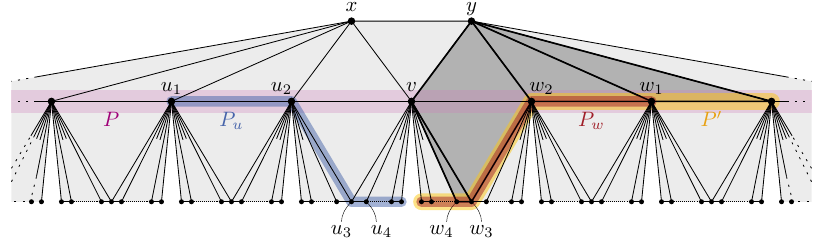}
        \caption{Illustrating the proof of \cref{prop:outerplanar-at-most-8}.}
        \label{fig:outerplanar-step}
    \end{figure}

    Next, we extend the coloring $\psi$ to the neighbors of $w_2$ (the case of $u_2$ is symmetric).
    Consider the path $P'$ obtained by gluing $P_w$ and a subpath of $P$ at $w_1$, see \cref{fig:outerplanar-step}, and observe that coloring $\psi$ is good on $P'$.
    Moreover, $w_2$ is the unique vertex on $P'$ with two colored neighbors ($v$ and $y$) outside $P'$.
    In other words, we are in the same situation as before, with $w_2$ taking the role of $v$ and $P'$ taking the role of $P$.
    By \cref{claim:outerplanar} we can extend the coloring $\psi$ to the neighbors of $w_2$ on $L_{i+1}$, and afterwards proceed with $w_1$ and so forth, to eventually color all neighbors of $P$.
\end{proof}

\section{Strong Odd Colorings for Small Treewidth}
\label{sec:treewidth}

In this section, we shall give an upper bound on the strong odd chromatic number for any graph $G$ in terms of its treewidth.
In particular, any graph class with bounded treewidth has bounded strong odd chromatic number.

\begin{theorem}\label{thm:treewidth-main}
    There exists a function~$f$ such that for every graph~$G$ we have \[\chiso(G) \leq f(\tw(G)).\]
\end{theorem}

In order to prove \cref{thm:treewidth-main}, we shall proceed by induction on the treewidth, showing a stronger statement (cf.~\cref{prop:treewidth-fine} below).
As we will need it in later sections, we consider strong odd colorings of \emph{directed} graphs. 
In fact, let us (for the rest of \cref{sec:treewidth,sec:row-treewidth}) identify every undirected graph~$G$ with the directed graph obtained by adding the directed edges $\vec{uv}$ and $\vec{vu}$ for each undirected edge $uv$.
As any strong odd coloring of the obtained directed graph is a strong odd coloring of~$G$, it suffices to consider directed graphs.

\begin{proposition}\label{prop:treewidth-fine}
    Let $k \geq 0$ and $\ell,m \geq 1$ be integers.
    There exists a constant $f_1(k,\ell,m)$ such that for every undirected $k$-tree~$G$, every collection $\vec{H}_1,\ldots,\vec{H}_\ell$ of directed subgraphs of~$G$, and every collection~$M_1, \ldots, M_m$ of subsets of~$V(G)$, there is a coloring $\Psi \colon V(G) \to [f_1(k,\ell,m)]$ of~$G$ with the following properties:
    \begin{enumerate}[(T1)]
        \item $\Psi$ is a proper coloring of $G$.\label{enum-tw:proper}
        \item $\Psi$ restricts to a strong odd coloring on the directed graph~$\vec{H}_i$ for every $i \in [\ell]$.\label{enum-tw:strong-odd}
        \item $\Psi$ is strong odd on the set~$M_j$ for every~$j \in [m]$.\label{enum-tw:marks-odd}
    \end{enumerate}
\end{proposition}

Before we prove \cref{prop:treewidth-fine}, let us deduce one of its consequences. Note that it might look odd, that we only use \cref{lem:treewidth-clique-coloring} inside the proof of \cref{prop:treewidth-fine}, but we will use it assuming \cref{prop:treewidth-fine} by induction.

\begin{lemma}
    \label{lem:treewidth-clique-coloring}
    If \cref{prop:treewidth-fine} holds for a fixed integer~$k \geq 0$ and all integers~$\ell,m \geq 1$, then the following statement holds as well.
    
    \smallskip
    
    There exists a constant~$g_1(k)$ such that for every $k$-tree~$G$ and every subset~$S$ of the $(k+1)$-cliques in $G$, there is a coloring $\sigma \colon S \to [g_1(k)]$ of~$S$ with the following properties:
    \begin{enumerate}[(T'1)]
     \item For every vertex~$v$ of~$G$, the coloring~$\sigma$ is strong odd on the set of~$(k+1)$-cliques containing~$v$, and
     \label{enum-clique-coloring:vertex_in_odd_number_of_cliques}
     \item every color class has odd size.\label{enum-clique-coloring:odd_sized_color_classes}
    \end{enumerate}
\end{lemma}
\begin{proof}
    Consider the construction sequence of the $k$-tree~$G$.
    In each step, a new vertex~$v$ is attached to all vertices of an existing $k$-clique~$C$, thereby forming a new $(k+1)$-clique~$Q_v$ of~$G$.
    We say that the vertices in~$C$ are the \emph{parents} of~$v$ and that~$v$ is a \emph{child} of every vertex in~$C$.
    The vertex~$v$ \emph{represents} the $(k+1)$-clique~$Q_v$.
    Note in particular that each $(k+1)$-clique is represented by exactly one vertex and no vertex represents several cliques.

    Given a vertex-coloring~$\psi$ of~$G$, we obtain the \emph{induced coloring}~$\sigma$ of the $(k+1)$-cliques in~$S$ by assigning each clique~$Q_v$ the color~$\psi(v)$ of the vertex~$v$ that represents it.
    In what follows, we will construct a directed subgraph~$\vec{H}$ of~$G$ and a subset~$V_S$ of the vertices~$V(G)$ such that every proper coloring~$\psi$ of~$G$ that restricts to a strong odd coloring on the directed graph~$\vec{H}$ and is strong odd on the set~$V_S$ induces a coloring~$\sigma$ of~$S$ that satisfies~(T'\ref{enum-clique-coloring:vertex_in_odd_number_of_cliques})-(T'\ref{enum-clique-coloring:odd_sized_color_classes}), see \cref{fig:tw-clique-coloring} for an example.

    \begin{figure}
    \centering
    \begin{subfigure}[t]{0.5\textwidth}
        \centering
        \includegraphics[page=2]{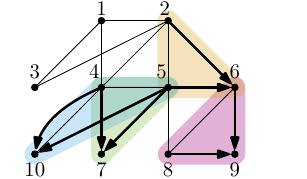}
        \caption{}
        \label{fig:example-tw-cliques}
    \end{subfigure}\hfill
    \begin{subfigure}[t]{0.5\textwidth}
        \centering
        \includegraphics[page=3]{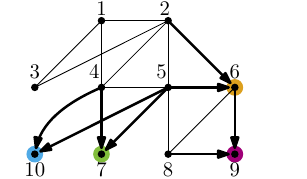}
        \caption{}
        \label{fig:example-tw-cliques-directed-graph}
    \end{subfigure}
    \caption{(\subref{fig:example-tw-cliques}) A~$2$-tree~$G$. The labels of the vertices correspond to their insertion order in the construction sequence of~$G$. The cliques of~$S$ are colored. (\subref{fig:example-tw-cliques-directed-graph}) The derived directed subgraph~$\vec{H}$ of $G$ is represented with thick lines. The colored vertices correspond to~$V_S$.}
    \label{fig:tw-clique-coloring}
    \end{figure}

    \begin{claim}
        There exists a directed subgraph~$\vec{H}$ of~$G$ such that if~$\psi$ is proper on~$G$ and strong odd on~$\vec{H}$, then    
        the induced coloring~$\sigma$ of~$S$ satisfies (T'\ref{enum-clique-coloring:vertex_in_odd_number_of_cliques}).
    \end{claim}
    \begin{claimproof}
    Consider the directed graph~$\vec{H}$ on the vertices of~$G$ where every vertex~$p \in V(G)$ is connected with a directed edge to each of its children~$v \in V(G)$ representing a clique in~$S$, i.e.,
    $
    E(\vec{H}) = \set{\vec{pv} \given \text{$p$ is a parent of~$v$ in~$G$ and $Q_v \in S$}}.
    $

    Suppose~$\psi$ is proper on~$G$ and strong odd on the directed graph~$\vec{H}$. 
    Let $p \in V(G)$ and let~$S(p) = \set{Q \in S \given p \in Q}$ denote the set of $(k+1)$-cliques of~$S$ containing~$p$.
    We need to show that $\sigma$ is strong odd on the set~$S(p)$, i.e., for every color~$c$ the cardinality~$\abs{S(p) \cap \sigma^{-1}(c)}$ is either odd or zero.
    The vertex~$p$ might represent a $(k+1)$-clique which we then denote by~$Q_p$.
    Observe that every $(k+1)$-clique~$Q_v$ in~$S(p)$, apart from the potential clique~$Q_p$, is represented by a child~$v \in N^+(\vec{H},p)$ of~$p$.
    Thus, as $\psi$ is strong odd on~$N^+(\vec{H},p)$, $\sigma$ is strong odd on $S(p)-Q_p$.
    As the coloring~$\psi$ is proper, no child of~$p$ is colored in~$\psi(p)$.
    Therefore, the clique~$Q_p$ (if it exists) has a different color from all other cliques in~$S(p)$ and~$\sigma$ is strong odd on the set~$S(p)$.
    \end{claimproof}

    \begin{claim}
        There exists a set~$V_S \subseteq V(G)$ such that  if~$\psi$ is strong odd on~$V_S$, then the induced coloring~$\sigma$ of~$S$ satisfies~(T'\ref{enum-clique-coloring:odd_sized_color_classes}).
    \end{claim}
    \begin{claimproof}
    Let~$V_S \subseteq V(G)$ be the set of vertices of~$G$ representing $(k+1)$-cliques in~$S$.
    Recall that there is a one-to-one-correspondence between cliques in~$S$ and vertices in~$V_S$.
    We therefore obtain for every color~$c$
    \[
        \abs{S \cap \sigma^{-1}(c)} = \abs{V_S \cap \psi^{-1}(c)}.
    \]
    Thus, if the vertex-coloring~$\psi$ is strong odd on the set~$V_S$, then the induced coloring~$\sigma$ of~$S$ satisfies~(T'\ref{enum-clique-coloring:odd_sized_color_classes}).
    \end{claimproof}
    
    If \cref{prop:treewidth-fine} holds for every $k$-tree, then there exists a vertex-coloring~$\psi$ with $g_1(k) = f_1(k,1,1)$ colors that is proper on~$G$, and strong odd on the directed graph~$\vec{H}$ and the set~$V_S$. 
    By the above, the induced coloring~$\sigma$ has the desired properties.
\end{proof}

\begin{proof}[Proof of \cref{prop:treewidth-fine}]
    We proceed by induction on the treewidth of~$G$ which we denote by~$k$. 
    If~$k=0$, the graphs~$G$ and~$\vec{H}_1,\ldots,\vec{H}_\ell$ contain no edges.
    In particular, any vertex-coloring is proper on~$G$ and strong odd for each $\vec{H}_i$, i.e., fulfills (T\ref{enum-tw:proper}) and (T\ref{enum-tw:strong-odd}).
    It therefore suffices to construct a coloring of~$G$ that satisfies~(T\ref{enum-tw:marks-odd}).

    Let~$M = \bigcup_{j \in [m]} M_j$ be the vertices of~$G$ that are contained in some set~$M_j$.
    For every subset $J$ of $[m]$, we define $V_J = \bigcap_{j \in J} M_j - \bigcup_{j \in [m] - J} M_j$, see \cref{fig:tw-zero-partition}.
    \begin{figure}
        \centering
        \includegraphics[page=2]{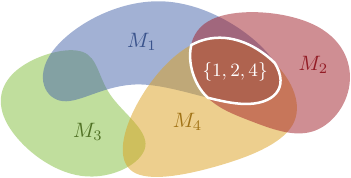}
        \caption{A visualization of a $0$-tree with four sets~$M_1, \dots, M_4$. The colored cells of the drawing correspond to the subsets~$J \subseteq [4]$. The border of the cell corresponding to~$\set{1,2,4}$ is highlighted in white.}
        \label{fig:tw-zero-partition}
    \end{figure}
    The sets~$V_J$ partition~$M$ into $2^m$~sets.
    Every set~$V_J$ is either completely contained in~$M_j$ or disjoint from~$M_j$ for every $j \in [m]$.
    We now construct a vertex-coloring of~$M$ that is strong odd on every set~$M_j$.
    For each set~$V_J$, we color all vertices in~$V_J$ with one color if~$\abs{V_J}$ is odd.
    If~$\abs{V_J}$ is even, all vertices in~$V_J$ except one get the same color.
    Using a different color palette on every set~$V_J$ yields a vertex-coloring~$\Psi$ with at most $f_1(0,\ell,m) \coloneqq 2^{m+1}$ colors that fulfills~(T\ref{enum-tw:marks-odd}).
    
    \smallskip
    
    Now let $k \geq 1$.
    Let~$L_1, \dots, L_{s}$ be a BFS-layering of the $k$-tree~$G$.
    For each layer~$L_d$ we construct a vertex-coloring~$\Phi_d$.
    The union of the vertex-colorings~$\Phi_d$ yields a coloring~$\Phi$ of~$G$.
    We will show that~$\Phi$ is proper and strong odd on the directed graphs~$\vec{H}_1, \dots, \vec{H}_{\ell}$.
    Yet, it might not be strong odd on the sets~$M_1, \dots, M_m$. 
    The desired coloring~$\Psi$ arises from~$\Phi$ by a small modification.
    
    The coloring~$\Phi_{d}$ is the product of four vertex-colorings~$\varphi_1,\varphi_2,\varphi_3$ and~$\varphi_4$ of~$L_{d}$. 
    The intuition is the following:
    Within a directed graph~$\vec{H}_i$, a vertex~$v \in L_d$ can have out-neighbors in three different layers, namely in~$L_{d-1},L_{d}$ and~$L_{d+1}$. 
    Let $L_t(v)$ denote the out-neighbors of~$v$ in layer~$L_t$ for~$t \in \set{d-1,d,d+1}$.
    The coloring~$\varphi_4$ will ensure that all out-neighbors of~$v$ of the same color in~$\Phi$ belong to the same layer, i.e., the coloring~$\Phi$ restricts to a strong odd coloring on~$\vec{H}_i$ if it is strong odd on each of the sets~$L_t(v)$ for every vertex~$v$.
    From the coloring~$\varphi_1$ we will conclude that the coloring~$\Phi$ is strong odd on~$L_d(v)$ and on the out-neighbors~$L_{d+1}(v) \cap C$ within~$L_{d+1}$ that belong to one component~$C$ of~$G[L_{d+1}]$. 
    The role of~$\varphi_2$ and~$\varphi_3$ is to ensure that the coloring~$\Phi$ is strong odd on~$L_{d+1}(v)$, which is the union of the sets $L_{d+1}(v) \cap C$ for all components~$C$ of the $(d+1)$-th layer.
    From the fact that~$\Phi$ is proper, we will conclude that the coloring is strong odd on the set~$L_{d-1}(v)$.

    We now construct the coloring~$\Phi_{d}$ of the $d$-th layer for~$d \geq 2$.
    Recall that each layer has treewidth at most~$k-1$ by Property~(B\ref{property:layer_tw_k-1}). 
    Thus, there exists a proper $k$-coloring~$\chi_{d-1}$ of the $(d-1)$-th layer with colors in~$[k]$.
    The vertex-coloring~$\chi_{d-1}$ will only be used for the construction of~$\Phi_d$; the coloring~$\Phi_{d-1}$ of the $(d-1)$-th level will not be altered.

    \proofsubparagraph{Construction of~$\bm{\varphi_1}$.}
    For every component~$C$ of~$G[L_{d}]$, the set of vertices in~$L_{d-1}$ with a neighbor in~$C$ forms a $k$-clique~$Q$ of the $(d-1)$-th layer by Property~(B\ref{property:parent_k-clique}).
    We call~$Q$ the \emph{parent-clique} of the vertices in the component~$C$ and say that these vertices are \emph{children} of~$Q$.
    Note in particular that~$Q$ may have children in several components of~$G[L_{d}]$, while every vertex of~$L_d$ has exactly one parent-clique.
    Let~$G_Q$ denote the graph induced by the children of~$Q$, see \cref{fig:tw_G_Q}.
    \begin{figure}
        \centering
        \includegraphics[page=2]{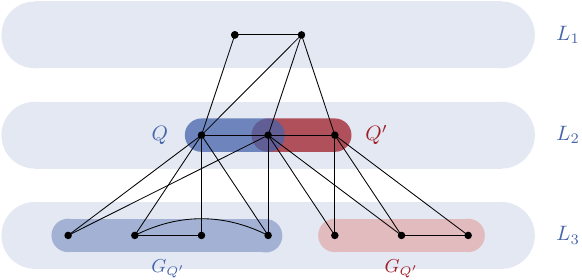}
        \caption{A $2$-tree. Two cliques~$Q,Q'$ in~$L_2$ are represented in blue and red. Their graphs~$G_{Q}$ and~$G_{Q'}$ in~$L_3$ are colored blue and red, respectively.}
        \label{fig:tw_G_Q}
    \end{figure}
    For every $k$-clique~$Q$ in the $(d-1)$-th layer, we will define a coloring~$\varphi_Q$ of $G_Q$.
    For a vertex $v \in L_d$ with parent-clique~$Q$, the color~$\varphi_1(v)$ corresponds to its color in~$\varphi_Q$, i.e., we will set $\varphi_1(v) \coloneqq \varphi_Q(v)$.
    The coloring~$\varphi_Q$ of~$G_Q$ is obtained as follows:
    
    For every~$i \in [\ell]$, let~$\vec{H}_{i,Q} \coloneqq \vec{H}_i \cap G_Q$ be the restriction of~$\vec{H}_i$ to the graph~$G_Q$ and for every~$j \in [m]$ we denote by~$M_{j,Q} \coloneqq M_j \cap V(G_Q)$ the subset of~$M_j$ restricted to the vertex-set of~$G_Q$.
    As~$Q$ is a $k$-clique in the $(d-1)$-th layer and~$\chi_{d-1}$ is a proper $k$-coloring, each of the $k$~vertices in~$Q$ has a different color in~$[k]$. 
    We denote by~$v_h$ the vertex of~$Q$ with $\chi_{d-1}(v_h) = h$ for every $h \in [k]$.
    For each $h \in [k]$ and $i \in [\ell]$, we let~$N_{i,h,Q} \coloneqq N^+(\vec{H_i},v_h) \cap V(G_Q)$ be the out-neighbors of~$v_h$ in~$\vec{H}_i$ that are in~$G_Q$. 
    By Property~(B\ref{property:layer_tw_k-1}), the graph~$G_Q$ has treewidth at most~$k-1$, as it is a subgraph of~$G[L_{d+1}]$.
    In particular, $G_Q$ is a subgraph of a $(k-1)$-tree~$G_Q'$.
    
    We obtain the coloring~$\varphi_Q$ of $G_Q$ as a subgraph of $G_Q'$ by induction on $G_Q'$. The coloring~$\varphi_Q$ has then the following properties:
    \begin{claim}
    \label{claim:properties_varphi_Q}
        The vertex-coloring~$\varphi_Q$ has the following properties:
        \begin{enumerate}[(a)]
            \item\label{claim_properties_varphi_Q_is_proper} $\varphi_Q$ is a proper vertex-coloring of~$G_Q$
            \item\label{claim_properties_varphi_Q_is_strong_odd_on_H_i} $\varphi_Q$ restricts to strong odd colorings on the directed graphs $\vec{H}_{1,Q}, \dots, \vec{H}_{\ell,Q}$
            \item\label{claim_properties_varphi_Q_is_strong_odd_on_M_i} $\varphi_Q$ is strong odd on the $m+\ell\cdot k$ subsets~$M_{1,Q}, \dots, M_{m,Q}$ and~$N_{i,h,Q}$ with $i \in [\ell], h \in [k]$
            
            \item\label{claim_properties_varphi_Q_num_colors} $\varphi_Q$ only uses colors in $[f_1(k-1,\ell,m+k\ell)]$.
        \end{enumerate}
    \end{claim}

    \proofsubparagraph{Construction of~$\bm{\varphi_2}$.}
    We assign a \emph{type} to each~$k$-clique~$Q$ of the $(d-1)$-th layer based on the coloring~$\varphi_Q$ defined above.
    The type is a $\mathsf{0}\,$-$\mathsf{1}$ matrix~$A_Q$ with $m+\ell\cdot k$ rows and $f_1(k-1,\ell,m+\ell k)$ columns defined as follows: 
    \begin{itemize}
        \item For $c \in [f_1(k-1,\ell,m+\ell k)]$, column~$c$ corresponds to color~$c$ in the coloring $\varphi_Q$ of $G_Q$.
            
        \item For $j \in [m]$, row~$j$ corresponds to the set~$M_{j,Q} = M_j \cap V(G_Q)$.
        The matrix~$A_Q$ has a $\mathsf{1}$ in row $j$ and column $c$ if and only if at least one vertex in~$M_{j,Q}$ in~$G_Q$ has color~$c$ in~$\varphi_Q$.
            
        \item For $h \in [k]$, let~$v_h$ be the vertex of~$Q$ with $\chi_{d-1}(v_h) = h$. 
        For $i \in [\ell]$ and $h \in [k]$, row~$m + (i-1)k + h$ corresponds to the vertex~$v_h$ in~$Q$ and the induced subgraph $\vec{H}_{i,Q}$ of $\vec{H}_i$ in $G_Q$.
        Here, $A_Q$ has a~$\mathsf{1}$ in row~$m + (i-1)k + h$ and column $c$ if and only if at least one neighbor~$u \in N_{i,h,Q} = N^+(\vec{H}_{i,Q},v_h) \cap V(G_Q)$ of~$v_h$ has color~$c$ in~$\varphi_Q$.
    \end{itemize}
    Types have a synchronizing effect in the following sense:
    
    \begin{claim}
        \label{claim:types}
        Let~$Q$ and~$Q'$ be two $k$-cliques of the $(d-1)$-th layer.
        If $Q$ and~$Q'$ are of the same type, then the following holds:
        \begin{enumerate}[(a)]
            \item\label{claim:types_k_cliques_of_same_type_M_j} For every~$j \in [m]$ and every color~$c$ of~$\varphi_Q$, we have $|M_{j,Q} \cap \varphi_Q^{-1}(c)|=0$ if and only if $|M_{j,Q'} \cap \varphi_{Q'}^{-1}(c)|=0$.
            \item\label{claim:types_k_cliques_of_same_type_N_i} If there is a vertex $v_h \in L_{d-1}$ with $\chi_{d-1}(v_h) = h$ that is part of~$Q$ and~$Q'$, then for every color~$c$ of~$\varphi$, we have $\abs{N_{i,h,Q} \cap \varphi_Q^{-1}(c)} =0$ if and only if $|N_{i,h,Q'} \cap \varphi_{Q'}^{-1}(c)|=0$.
        \end{enumerate}
    \end{claim}
    See \cref{fig:tw_types} for an illustration.
    
    \begin{figure}
        \savebox{\imagebox}{\includegraphics[page=3]{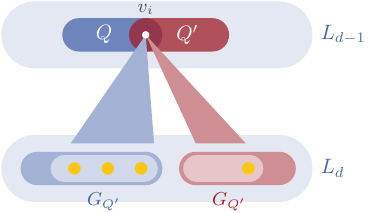}}
        \centering
        \begin{subfigure}[b]{0.4\textwidth}
        \centering\raisebox{\dimexpr.5\ht\imagebox-.5\height}{\includegraphics[page=2]{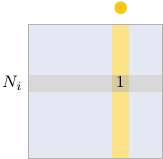}}
        \caption{}
        \label{fig:tw_types_matrix}
        \end{subfigure}\hfill
        \begin{subfigure}[b]{0.6\textwidth}
        \centering\usebox{\imagebox}
        \label{fig:tw_types_interaction}
        \caption{}
        \end{subfigure}
        \caption{The cliques~$Q,Q'$ in~$L_{d-1}$ are of the same type~$A$ represented in (\subref{fig:tw_types_matrix}). Both contain the vertex~$v_i$. 
        The matrix~$A$ has a~$1$ in the position representing color~\textcolor{lipicsYellow}{$\bullet$} and the vertex~$v_i$. Thus, $v_i$ has a neighbor in \textcolor{lipicsYellow}{$\bullet$} both in~$G_Q$ and~$G_Q'$.}
        \label{fig:tw_types}
    \end{figure}

    \proofsubparagraph{Construction of~$\bm{\varphi_3}$.}
    For each of the above-defined types~$A$, we construct a coloring~$\sigma_A$ of the $k$-cliques of the $(d-1)$-th layer of type~$A$ with the following properties:
    \begin{itemize}
        \item For every vertex~$v \in L_{d-1}$, the coloring~$\sigma_A$ is strong odd on the set of $k$-cliques of type~$A$ which contain~$v$.
        \item Every color class has odd size.
        \item The coloring uses at most~$g_1(k-1)$ colors (where~$g_1(k-1)$ only depends on~$k-1$).
    \end{itemize}
    Such a coloring exists by \cref{lem:treewidth-clique-coloring}.

    \proofsubparagraph{Definition of~$\bm{\Phi_d}$ and Properties of~$\bm{\Phi}$.}
    We now define the vertex-coloring~$\Phi_d$ of the $d$-th layer.

    First suppose that $d \geq 2$.
    Let~$v \in L_d$, let~$Q$ be its parent-clique in~$L_{d-1}$ and let~$A$ denote the type of~$Q$.
    We define $\Phi_d(v) = (\varphi_1(v),\varphi_2(v),\varphi_3(v),\varphi_4(v))$
    where
    \begin{itemize}
    \item $\varphi_1(v) = \varphi_Q(v)$ is the color of~$v$ under the partial coloring~$\varphi_Q$ of~$G[L_d]$,
    \item $\varphi_2(v) = A$ is the type of~$Q$,
    \item $\varphi_3(v) = \sigma_A(Q)$ is the color of the parent-clique~$Q$ in the coloring~$\sigma_A$ of all $k$-cliques of the~$(d-1)$-th layer of type~$A$, and
    \item $\varphi_4(v) = d \bmod 3$ is the index of the layer~$L_d$ taken modulo~$3$.
    \end{itemize}
    If~$d=1$, the vertices of~$L_d$ induce a $k$-clique by Property~(B\ref{property:root_k-clique}).
    Let~$\varphi_1$ be any proper $k$-coloring of~$L_1$.
    For a vertex~$v \in L_1$, we define~$\Phi_1(v) = (\varphi_1(v),-1,-1,1\bmod 3)$.
    Note in particular that no color used on~$L_1$ is used on any other layer.
    
    The coloring~$\Phi$ of~$G$ corresponds to the union of the colorings~$\Phi_d$ of the different layers, i.e., for every vertex~$v$ of the $d$-th layer, we have $\Phi(v) = \Phi_d(v)$.

    Note that the number of colors used by~$\Phi$ only depends on~$k,\ell$ and~$m$.
    This follows from the fact that the colorings~$\varphi_Q$ use at most $f_1(k-1,\ell, m+k\ell)$ colors by \cref{claim:properties_varphi_Q}\ref{claim_properties_varphi_Q_num_colors}, there are at most~$2^{(m+\ell k) f_1(k-1,\ell,m+\ell k)}$ different types of parent-cliques, the colorings~$\sigma_A$ use at most~$g_1(k-1)$ colors, and only~$k$ more colors are used for the first layer.

    We now observe that~$\Phi$ is proper, i.e., satisfies (T\ref{enum-tw:proper}).
    Consider two adjacent vertices~$u$ and~$v$.
    If~$u$ and~$v$ belong to different layers, their colors in~$\Phi$ differ due to the last entry~$\varphi_4$. 
    Otherwise, they belong to the same layer~$L_d$. 
    It suffices to observe that~$\varphi_1$ is a proper coloring of~$L_d$.
    This clearly holds for the first layer.
    For a different layer, $\varphi_1$ is the union of colorings~$\varphi_Q$ of the different components. 
    As every~$\varphi_Q$ is proper by~\cref{claim:properties_varphi_Q}\ref{claim_properties_varphi_Q_is_proper}, so is~$\varphi_1$.

    It remains to show that~$\Phi$ is strong odd on every directed graph~$\vec{H}_i$, i.e., satisfies (T\ref{enum-tw:strong-odd}). 
    Let $v \in V(\vec{H}_i)$ be a vertex and let~$S \coloneqq N^+(\vec{H}_i,v) \cap \Phi^{-1}(c)$ be the out-neighbors of~$v$ in color~$c$.
    Let $L_d$ be the layer the vertex~$v$ belongs to.
    We need to show that the cardinality of~$S$ is zero or odd.
    Due to the definition of the fourth entry of~$\Phi$, the set~$S$ is contained in one of the layers~$L_{d-1},L_d,L_{d+1}$ and we distinguish these three cases:
    
    \textbf{Case 1.} If~$S$ is contained in~$L_{d-1}$, then~$S$ is part of the parent-clique of~$v$, see \cref{fig:tw-neighbors_L_d-1}. 
    It follows that~$S$ has size one, as the coloring~$\Phi$ is proper.
    
    \textbf{Case 2.} Suppose~$S$ is contained in~$L_d$.
    If~$d=1$, then~$S$ induces a clique in~$G$. 
    It follows that~$S$ has size~$0$ or~$1$ as~$\Phi$ is a proper coloring.
    We may therefore assume that~$d\geq 2$.
    Observe that the vertices of~$S$ together with the vertex~$v$ are all contained in a single~$G_Q$ for some $k$-clique~$Q \subseteq L_{d-1}$, see \cref{fig:tw-neighbors_L_d}.
    Note in particular that $S = N^+(\vec{H}_i,v) \cap \Phi^{-1}(c) \cap V(G_Q)$.
    As the coloring~$\varphi_Q$ is strong odd on the restriction of~$\vec{H}_i$ to~$G_Q$ by \cref{claim:properties_varphi_Q}\ref{claim_properties_varphi_Q_is_strong_odd_on_H_i}, the size of~$S$ is zero or odd.

    \textbf{Case 3.} $S$ is contained in~$L_{d+1}$.
    Let~$A$ be the type that corresponds to the second entry $c_2$ of the color~$c$.
    We denote by~$Q_1, \dots, Q_t$ the $k$-cliques in~$G[L_d]$ that contain~$v$ and are of type~$A$.
    The vertex-set of~$S$ splits into~$t$ sets~$S_1, \dots, S_t$ where for each~$i\in[t]$ the vertices in~$S_i$ are children of the $k$-clique~$Q_i$, see \cref{fig:tw-neighbors_L_d+1}.
    
    \begin{figure}
        \centering
        \begin{subfigure}[b]{0.23\textwidth}
        \centering
        \includegraphics[page=2]{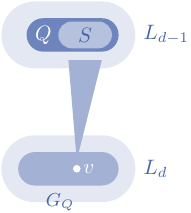}
        \caption{}
        \label{fig:tw-neighbors_L_d-1}
        \end{subfigure}\hfill
        \begin{subfigure}[b]{0.25\textwidth}
        \centering
        \includegraphics[page=3]{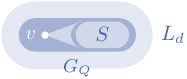}
        \caption{}
        \label{fig:tw-neighbors_L_d}
        \end{subfigure}\hfill
        \begin{subfigure}[b]{0.5\textwidth}
        \includegraphics[page=4]{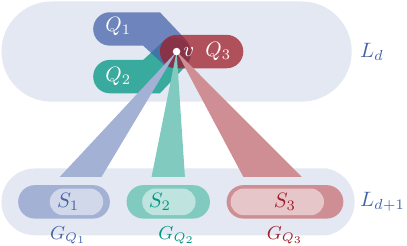}
        \caption{}
        \label{fig:tw-neighbors_L_d+1}
        \end{subfigure}
        \caption{The out-neighbors~$S$ of a vertex~$v \in L_d$ of the same color in~$\Phi$. They are either all contained in~$L_{d-1}$ (\subref{fig:tw-neighbors_L_d-1}), in~$L_d$ (\subref{fig:tw-neighbors_L_d}) or in~$L_{d+1}$ (\subref{fig:tw-neighbors_L_d+1}).}
        \label{fig:tw-neighbors}
    \end{figure}
    
    Due to the definition of the colorings~$\varphi_{Q_i}$, see \cref{claim:properties_varphi_Q}\ref{claim_properties_varphi_Q_is_strong_odd_on_M_i}, the size of each~$S_i$ is zero or odd.
    As for all $i\in[t]$ the $k$-clique~$Q_i$ is of the same type and contains the vertex~$v$, it follows from \cref{claim:types}\ref{claim:types_k_cliques_of_same_type_N_i} that either the size of every set~$S_i$ is zero, or the size of every set~$S_i$ is odd.
    In order to show that~$S$ is empty or has odd size, it thus suffices to prove that the number of sets~$S_1, \dots, S_t$ is zero or odd, i.e. $t=0$ or $t$ is odd.
    Recall that the third entry~$c_3$ of color~$c$ is derived from the coloring~$\sigma_A$ of the~$k$-cliques of the $j$-th layer of type~$A$.
    In particular, we obtain $\sigma_A(Q_i) = c_3$ for all $k$-cliques~$Q_i$.
    By definition of~$\sigma_A$, the number of $k$-cliques~$Q_i$ of type~$A$ that contain~$v$ is either zero or odd.
    Thus $t=0$ or $t$ is odd.

    Therefore,~$\Phi$ is a proper coloring of~$G$ that is strong odd on every directed graph~$\vec{H}_i$, i.e. satisfies (T\ref{enum-tw:proper}) and (T\ref{enum-tw:strong-odd}).

    \proofsubparagraph{Construction and Properties of~$\bm{\Psi}$.}
    Now it is time for the final modification of~$\Phi$ so that the resulting coloring~$\Psi$ still satisfies (T\ref{enum-tw:proper}) and (T\ref{enum-tw:strong-odd}), but furthermore (T\ref{enum-tw:marks-odd}).
    Hence, we need to make sure that among the vertices in~$M_j$ (for every $j \in [m]$) each color~$c$ appears either not at all, or an odd number of times. First, we establish two more properties of~$\Phi$.

    \begin{claim}
        \label{claim:intersection_M_j_with_layer_is_strong_odd}
        For every layer~$L_d$ and every set~$M_j$, the coloring~$\Phi$ is strong odd on the set~$L_d \cap M_j$.
    \end{claim}
    \begin{claimproof}
        If $d = 1$ the claim holds since the first layer is a clique by Property~(B\ref{property:root_k-clique}) and~$\Phi$ is proper.
        Now suppose $d \geq 2$ and fix some color~$c = (c_1,c_2,c_3,c_4)$ of~$\Phi$.
        We need to show that the size of~$S \coloneqq L_d \cap M_j \cap \Phi^{-1}(c)$ is zero or odd.
        Let~$Q_1,\dots,Q_t$ be $k$-cliques of the $(d-1)$-th layer such that every vertex in~$M_j$ is a child of some~$Q_i$ and each~$Q_i$ has a child in~$M_j$. 
        For every~$i \in [t]$, let~$S_i$ denote the children of~$Q_i$ in~$S$, see \cref{fig:tw_Phi_strong_odd_on_M_j_cap_L}.
        Note that the sets~$S_i$ form a partition of~$S$.
        \begin{figure}
        \centering
            \includegraphics[page=1]{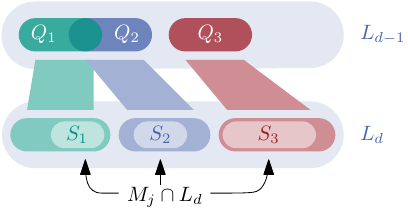}
            \caption{A visualization of the proof of \cref{claim:intersection_M_j_with_layer_is_strong_odd}. Here, each vertex of~$M_j \cap L_d$ is a child of one of the cliques~$Q_1,Q_2,Q_3$. The graphs~$G_{Q_i}$ are represented in the same color as their respective parent-cliques.}
            \label{fig:tw_Phi_strong_odd_on_M_j_cap_L}
        \end{figure}

        Recall that each of the colorings~$\varphi_{Q_i}$ is strong odd on the set~$M_{j,Q_i} = M_j \cap V(G_{Q_i})$ by \cref{claim:types}\ref{claim:types_k_cliques_of_same_type_M_j}.
        Thus, due to the first entry $c_1$ of~$c$, the size of each of the sets~$S_i = M_{j,Q_i} \cap \Phi^{-1}(c)$ is zero or odd.
        
        Due to the second entry $c_2$ of~$c$, all of the $k$-cliques~$Q_i$ are of the same type~$A$. 
        It therefore follows from \cref{claim:types}\ref{claim:types_k_cliques_of_same_type_M_j} that either all sets~$S_i$ are empty or all these sets have odd size.

        As the cardinality of the set~$S$ is the sum of the sizes of the sets~$S_i$ with $i \in [t]$, it remains to show that $t = 0$ or $t$ is odd.
        By definition of~$\Phi$, we have $\sigma_A(Q_i) = c_3$ for every $i \in [t]$. 
        It follows from \cref{claim:types}\ref{claim:types_k_cliques_of_same_type_M_j} that each $k$-clique~$Q \subseteq L_{d-1}$ of color~$c_3$ in $\sigma_A$ corresponds to one of the cliques~$Q_1, \dots, Q_t$ as such a $k$-clique~$Q$ is of type~$A$.
        Thus, the cliques~$Q_1, \dots, Q_t$ form a color class of~$\sigma_A$.
        As every color class of~$\sigma_A$ has odd size, we see that $t=0$ or~$t$ is odd.
    \end{claimproof}

    \begin{claim}
        \label{claim:M_j_on_every_relevant_layer}
        Let~$c$ be a color of~$\Phi$ and~$L_d$ a layer.
        If~$L_d \cap \Phi^{-1}(c) \neq \varnothing$ and $M_j \cap \Phi^{-1}(c) \neq \varnothing$, then $M_j \cap L_d \cap \Phi^{-1}(c) \neq \varnothing$, see \cref{fig:tw_intersection_color_M_j} for an illustration.
    \end{claim}
    \begin{claimproof}
        If~$d=1$ the claim holds since no color of~$L_1$ is used on any other layer and all vertices of~$L_1$ have distinct colors.
        We may therefore assume that~$d\geq 2$.
        As~$M_j \cap \Phi^{-1}(c) \neq \varnothing$ and $L \cap \Phi^{-1}(c) \neq \varnothing$, there exist $k$-cliques~$Q$ and~$Q'$ of type~$A$ such that~$Q$ has a child in~$M_j$ of color~$c$ and all children of~$Q'$ lie in~$L$.
        As~$Q$ and~$Q'$ are of the same type, it follows from~\cref{claim:types}\ref{claim:types_k_cliques_of_same_type_M_j} that~$Q'$ has a child in $M_j$ of color~$c$.
        Yet, all children of~$Q'$ lie in~$L$ and therefore $M_j \cap L \cap \Phi^{-1}(c) \neq \varnothing$.
    \end{claimproof}
    \begin{figure}
            \centering
            \includegraphics[page=2]{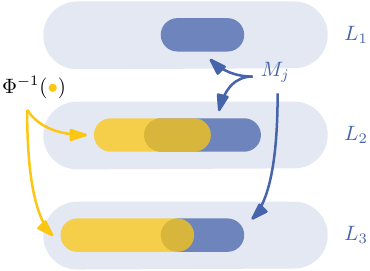}
            \caption{A visualization of the statement of \cref{claim:M_j_on_every_relevant_layer}. If a vertex of~$M_j$ is colored in~\textcolor{lipicsYellow}{$\bullet$} in the coloring~$\Phi$, then every layer on which color~\textcolor{lipicsYellow}{$\bullet$} appears, contains a vertex of~$M_j$ colored in~\textcolor{lipicsYellow}{$\bullet$}.}
            \label{fig:tw_intersection_color_M_j}
    \end{figure}
    
    Now, the coloring~$\Psi$ is derived from~$\Phi$ by splitting some color classes.
    For every color~$c$ of~$\Phi$, let~$\ell_c$ be the number of layers that color~$c$ appears on.
    If~$\ell_c$ is odd or zero, then we transfer color~$c$ to~$\Psi$ unaltered.
    Otherwise, we rename color~$c$ on exactly one of the layers it appears on into~$c'$.
    Thus, every color class of~$\Psi$ appears on an odd number of layers and~$\Psi$ uses at most twice as many colors as~$\Phi$.   

    Note that the coloring~$\Psi$ satisfies (T\ref{enum-tw:proper}) and (T\ref{enum-tw:strong-odd}) as these properties hold for~$\Phi$.
    It remains to show that~$\Psi$ is strong odd on every set~$M_j$.
    Let~$c$ be a color of~$\Psi$ and consider a set~$M_j$.
    Recall that for every layer~$L$, the size of~$\Psi^{-1}(c) \cap M_j \cap L$ is zero or odd by \cref{claim:intersection_M_j_with_layer_is_strong_odd}.
    Thus, the size of $M_j \cap \Psi^{-1}(c)$ is zero or odd if color~$c$ appears on none or an odd number of layers that contain vertices of~$M_j$.
    By definition of~$\Psi$, color~$c$ appears on an odd number of layers.
    Together with \cref{claim:M_j_on_every_relevant_layer}, we conclude that the number of layers~$L$ such that $M_j \cap L \cap \Psi^{-1}(c) \neq \varnothing$ is odd or zero.
    As this holds for every color~$c$, the coloring~$\Psi$ is strong odd on the set~$M_j$.
    It follows that~$\Psi$ satisfies (T\ref{enum-tw:proper})-(T\ref{enum-tw:marks-odd}).
\end{proof}

\section{Strong Odd Colorings for Small Row-Treewidth}
\label{sec:row-treewidth}

The aim of this section is to bound the strong odd chromatic number in the row-treewidth.

\begin{theorem}\label{thm:row-treewidth}
    There exists a function~$f$ such that for every graph~$G$ we have \[\chiso(G) \leq f(\rtw(G)).\]
\end{theorem}

This follows from a stronger statement (cf. \cref{prop:row-treewidth-fine} below).
Recall that a graph~$G$ has row-treewidth~$k$ if there exists a $k$-tree~$H$ and a path~$P$ such that $G \subseteq H \boxtimes P$. 
Thus, it suffices to construct a proper coloring of~$H \boxtimes P$ that restricts to a strong odd coloring on~$G$.

The basic idea resides in coloring each layer of~$H \boxtimes P$ individually, using the same color set only on every third layer.
The $i$-th layer $L_i$ of $H \boxtimes P$ corresponds to a copy of~$H$.
Thus, according to \cref{thm:treewidth-main}, there is a coloring of~$L_i$ which only uses few colors and restricts to a strong odd coloring on~$G \cap L_i$.
The coloring of~$H$ we obtain by independently coloring each layer (using the same color set only on every third layer) is proper on~$H$, yet, it may not be strong odd on~$G$ as we did not take inter-level-edges into account.
We will account for inter-level edges by considering their projections into each layer~$L_i$ and modeling these as directed graphs.

\begin{proposition}\label{prop:row-treewidth-fine}
    Let~$k,m \geq 0$ be integers.
    There exists a constant~$f_2(k,m)$ such that for every $k$-tree~$H$, every path~$P$, every directed subgraph~$\vec{G}$ of~$H \boxtimes P$ and every collection~$M_1, \dots, M_m$ of subsets of~$V(H \boxtimes P)$, there is a coloring~$\Psi\colon V(H \boxtimes P) \to [f_2(k,m)]$ of~$H \boxtimes P$ with the following properties:
    \begin{enumerate}[(R1)]
        \item \label{enum-rtw:proper} $\Psi$ is a proper coloring of~$H \boxtimes P$
        \item \label{enum-rtw:strong_odd} $\Psi$ restricts to a strong odd coloring on the directed subgraph~$\vec{G}$
        \item \label{enum-rtw:marks-odd} $\Psi$ is strong odd on the set~$M_j$ for every $j \in [m]$. 
    \end{enumerate}
\end{proposition}
\begin{proof}
    We first construct a coloring~$\Phi$ of~$H \boxtimes P$  that is the product of three colorings~$\varphi_1, \varphi_2, \varphi_3$ and fulfills (R\ref{enum-rtw:proper}) and (R\ref{enum-rtw:strong_odd}).
    A small modification then yields the desired coloring~$\Psi$.

    \proofsubparagraph{Construction of~$\bm{\varphi_1}$.}
    Let~$p_1, p_2, \dots, p_n$ be the order of vertices along the path~$P$.
    Recall that the vertices of~$H \boxtimes P$ can be partitioned into layers~$L_d = \set{(v,p_d) \given v \in V(H)}$. 
    Each layer corresponds to a copy of~$H$.
    We write $v_d = (v,p_d)$ for the vertex in the~$d$-th layer representing a vertex~$v \in V(H)$.
    For every layer~$L_d$, we define three directed subgraphs~$\vec{G}_d, \vec{G}_{d-1,d}$ and~$\vec{G}_{d+1,d}$, see \cref{fig:rtw-directed-graphs} for an example.
    These three graphs model the three different types of edges of~$\vec{G}$ incident to vertices in~$L_d$:
    the intra-level edges within~$L_d$ (represented by~$\vec{G}_d$) the inter-level edges between~$L_{d-1}$ and $L_{d}$ (represented by~$\vec{G}_{d-1,d}$), and the inter-level edges between~$L_{d+1}$ and~$L_{d}$ (represented by~$\vec{G}_{d+1,d}$).
    \begin{figure}
        \centering
        \begin{subfigure}[b]{0.55\textwidth}
        \centering
        \includegraphics[page=2]{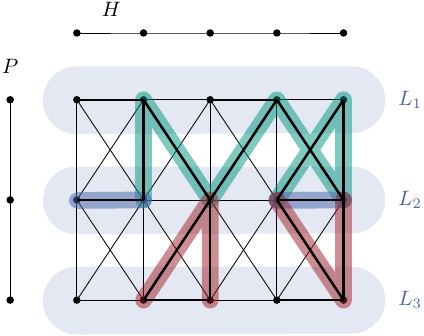}
        \caption{}
        \label{fig:rtw-layering}
        \end{subfigure}\hfill
        \begin{subfigure}[b]{0.45\textwidth}
        \centering
        \includegraphics[page=5]{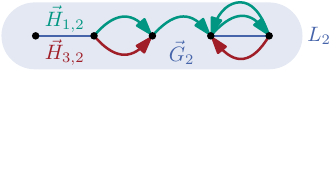}
        \caption{}
        \label{fig:rtw-projection}
        \end{subfigure}
        \caption{(\subref{fig:rtw-layering}) Here, the edges of the directed graph~$\vec{G}$ (represented with thick edges) are present in both directions. The edges of~$\vec{G}$ represented by~$\vec{G}_2$, $\vec{H}_{1,2}$, $\vec{H}_{3,1}$ are highlighted in blue, green and red respectively. (\subref{fig:rtw-projection}) The directed graphs~$\vec{G}_2$, $\vec{H}_{1,2}$, $\vec{H}_{3,1}$ are represented.}
        \label{fig:rtw-directed-graphs}
    \end{figure}
    \begin{itemize}
        \item $\vec{G}_d$ is the graph~$G[V(\vec{G}) \cap L_d]$ that corresponds to the subgraph of~$\vec{G}$ within~$L_d$,
        \item $\vec{G}_{d-1,d}$ is the graph on the vertices of~$L_{d}$ where two vertices~$u_{d} = (u,p_{d})$ and~$v_{d} = (v,p_{d})$ with~$u_{d} \neq v_{d}$ are connected with an edge~$\overrightarrow{u_{d} v_{d}}$ if and only if~$u_{d-1} v_{d} \in E(\vec{G})$,
        \item $\vec{G}_{d+1,d}$ is the graph on the vertices of~$L_d$ where two vertices~$u_d = (u,p_d)$ and~$v_d = (v,p_{d})$ with~$u_d \neq v_d$ are connected with an edge~$\overrightarrow{u_d v_d}$ if and only if~$u_{d+1} v_{d} \in E(\vec{G})$.
    \end{itemize}
    For each layer~$L_d$, let~$\gamma_d \colon L_d \to [f_1(k,3,m)]$ be the coloring given by \cref{prop:treewidth-fine} with respect to the $k$-tree~$H$, the directed graphs~$\vec{G}_{d-1,d}, \vec{G}_{d+1,d}$ and $\vec{G}_d$ and the restrictions~$M_1 \cap L_d, \dots, M_m \cap L_d$ of the sets~$M_i$ to layer~$L_d$.
    The coloring~$\varphi_1$ of a vertex~$v_d \in L_d$ corresponds to its color in $\gamma_d$.

    \proofsubparagraph{Construction of~$\bm{\varphi_2}$}
    We assign to each layer~$L_d$ a type based on the colors of~$\gamma_d$. 
    This type is a $\mathsf{0}\,$-$\mathsf{1}$ matrix~$A_d$ with $m$ rows and $3 \cdot f_1(k,3,m)$ columns defined as follows: 
    \begin{itemize}
        \item For $c \in [3 \cdot f_1(k,3,m)]$, column~$c$ corresponds to color~$c$ in the coloring~$\gamma_d$ of $L_d$.
        \item For $j \in [m]$, row~$j$ corresponds to the set~$M_j \cap V(L_d)$.
        The matrix~$A_d$ has a $\mathsf{1}$ in row $j$ and column $c$ if and only if at least one vertex in~$M_{j}$ in layer~$L_d$ has color~$c$ in~$\gamma_d$.
    \end{itemize}
    The color of a vertex in~$H \boxtimes P$ under~$\varphi_2$ corresponds to the type of its layer.
    
    \proofsubparagraph{Definition and Properties of~$\bm{\Phi}$}
    Now we define the coloring~$\Phi$.
    Let $v_d \in L_d$. 
    We define $\Phi(v_d) = (\varphi_1(v_d),\varphi_2(v_d),\varphi_3(v_d))$ where 
    \begin{itemize}
        \item $\varphi_1(v_d) = \gamma_d(v_d)$ is the color of~$v$ under the coloring~$\gamma_d$ of~$L_d$
        \item $\varphi_2(v_d) = A_d$ is the type of~$L_d$,
        \item $\varphi_3(v_d) = d \bmod 3$ is  the index of the layer~$L_d$  modulo~$3$.
    \end{itemize}
   Clearly,~$\Phi$ is a proper coloring of~$H \boxtimes P$ with at most~$f_1(k,3,m) \cdot 2^{m \cdot f_1(k,3,m)} \cdot 3$ colors, thus~$\Phi$ fulfills~(T\ref{enum-rtw:proper}).
    It remains to show that~$\Phi$ restricts to a strong odd coloring on~$\vec{G}$.
    Consider a vertex~$v_d \in V(H \boxtimes P)$.
    Its out-neighborhood~$N^+(\vec{G},v)$ in~$\vec{G}$ splits into three sets~$N_i \coloneqq N^+(\vec{G},v) \cap L_i$ with $i \in \set{d-1,d,d+1}$.
    A different color palette is used on each of the sets~$N_i$ due to the second entry $\varphi_2$ of~$\Phi$.
    It thus suffices to show that~$\Phi$ is strong odd on each of the sets~$N_i$.
    Note that we have 
    \[N_d = N^+(\vec{G}_d,v_d),\quad N_{d-1} \subseteq N^+(\vec{G}_{d,d-1},v_d) \cup \set{v_{d-1}},\quad N_{d+1} \subseteq N^+(\vec{G}_{d,d+1},v_d) \cup \set{v_{d+1}}.\]
    The coloring~$\Phi$ is clearly strong odd on~$N_d$, as~$\gamma_d$ is strong odd on~$N^+(\vec{G}_d,v_d)$ and neither~$\varphi_2$ nor~$\varphi_3$ change the color classes within a single layer.
    We now argue that~$\Phi$ is strong odd on~$N_{d-1}$.
    The vertex~$v_{d-1}$ is adjacent to every vertex in~$N^+(\vec{G}_{d,d-1},v_d)$. 
    As the coloring~$\Phi$ is proper, no vertex in~$N^+(\vec{G}_{d,d-1},v_d)$ has the same color as~$v_{d-1}$, i.e. $\Phi$ is strong odd on~$N_{d-1}$ if it is strong odd on~$N^+(\vec{G}_{d,d-1},v_d)$.
    Recall that~$\gamma_{d-1}$ restricts to a strong odd coloring on~$\vec{G}_{d,d-1}$.
    It follows that $\Phi$ is strong odd on~$N_{d-1}$.
    A similar argument shows that~$\Phi$ is strong odd on~$N_{d+1}$.

    Thus, $\Phi$ is a vertex-coloring of~$H \boxtimes P$ with $f_1(k,3,m) \cdot 2^{m \cdot f_1(k,3,m)} \cdot 3$ colors that fulfills (T\ref{enum-rtw:proper})-(T\ref{enum-rtw:strong_odd}).

    \proofsubparagraph{Construction of~$\bm{\Psi}$.}
    Due to the first entry~$\varphi_1$ of~$\Phi$, the coloring~$\Phi$ is strong odd on the sets~$L\cap M_j$ for every layer~$L$ and every $j \in [m]$.
    From the definition of types, we conclude the following.
    \begin{claim}
        Let~$c$ be a color of~$\Phi$, $j \in [m]$ and~$L$ a layer.
        If~$L \cap \Phi^{-1}(c) \neq \varnothing$ and $M_j \cap \Phi^{-1}(c) \neq \varnothing$, then $L \cap M_j \cap \Phi^{-1}(c) \neq \varnothing$.
    \end{claim}

    We observe that if every color appears on an odd number of layers, then $\Phi$ is strong odd on each of the sets~$M_j$.
    This can be achieved as follows.
    If a color~$c$ of~$\Phi$ appears on an even number of layers, we recolor all vertices of color~$c$ on one layer in a new color~$c'$.
    Proceeding in such a way for every color of~$\Phi$ yields the coloring~$\Psi$. 
    The coloring~$\Psi$ thus satisfies (R\ref{enum-rtw:marks-odd}).
    As~$\Psi$ still satisfies (R\ref{enum-rtw:proper})--(R\ref{enum-rtw:strong_odd}) and uses only two times as many colors as~$\Phi$, namely~$f_2(k,m) = 6 \cdot f_1(k,3,m) \cdot 2^{m \cdot f_1(k,3,m)}$, the result follows.
\end{proof}

\section{Strong Odd Colorings of Proper Minor-Closed Graph Classes}
\label{sec:proper_minor_closed}

The aim of this section is to bound the strong odd chromatic number in subgraphs of $(w,k,t)$-sums in~$w,k$ and~$t$.

\begin{theorem}\label{thm:wkt-sums}
    There exists a function~$f$ such that for every subgraph~$G$ of a $(w,k,t)$-sum we have \[\chiso(G) \leq f(w,k,t).\]
\end{theorem}

As every proper minor-closed graph class is obtained by taking subgraphs of $(w,k,t)$-sums (cf. \cref{thm:minor-closed-sums}), this implies our main result~\cref{thm:minor-closed}.

In order to show that the strong chromatic number is bounded in terms of treewidth, we considered BFS-layerings. 
If~$L_1, \dots, L_d$ is a BFS-layering of a $k$-tree~$G$ (cf. \cref{obs:bfs-layering}), the neighbors of every component~$C$ of a layer~$L_i$ form a $k$-clique~$Q$ in~$L_{i-1}$.
The strong odd coloring now arises in part by coloring the components~$C$ independently.

To bound the strong odd chromatic number of subgraphs of $(w,k,t)$-sums, we consider natural layerings instead. 
Here, $(k,t)$-summands play a role very similar to the graphs~$C \cup Q$ in the case of bounded treewidth.

\begin{lemma}
\label{lem:summand-fine}
    Let~$k, m, t \geq 0$ be integers.
    There exists a constant~$f_3(k,t,m)$ such that for every $(k,t)$-summand~$F$ and every directed subgraph~$\vec{G}$ of~$F$ and every collection~$M_1, \dots, M_m$ of subsets of~$V(F)$, there is a coloring~$\Psi\colon V(F) \to [f_3(k,t,m)]$ of~$F$ with the following properties:
    \begin{enumerate}[(F1)]
        \item \label{enum-summand:proper} $\Psi$ is a proper coloring of~$F$,
        \item \label{enum-summand:strong_odd} $\Psi$ restricts to a strong odd coloring on the directed subgraph~$\vec{G}$,
        \item \label{enum-summand:marks-odd} $\Psi$ is strong odd on the set~$M_j$ for every $j \in [m]$. 
    \end{enumerate}
\end{lemma}
\begin{proof}
    Let~$H$ be a $k$-tree and~$P$ be a path such that $F = (H \boxtimes P) + K_t$.
    Let~$v_1, \dots, v_t$ be the vertices of~$K_t \subseteq F$.
    For every vertex~$v_i$, let~$N_i \coloneqq N^+(\vec{G},v_i)-K_t$ denote the neighbors of~$v_i$ in $\vec{G} \cap (H \boxtimes P)$.
    Consider the coloring~$\Phi$ of~$H \boxtimes P$ with respect to the directed subgraph~$\vec{G}$, the sets~$M_j \cap V(H \boxtimes P)$ for $j \in [m]$ and the sets~$N_i$ for $i \in [t]$ we obtain by \cref{prop:row-treewidth-fine}. 
    Using $t$ new colors on the vertices of~$K_t$, we extend~$\Phi$ to a coloring~$\Psi$ of~$F$. 
    As~$\Phi$ satisfies (R\ref{enum-rtw:proper})-(R\ref{enum-rtw:marks-odd}), $\Psi$ satisfies (F\ref{enum-summand:proper})-(F\ref{enum-summand:marks-odd}).
\end{proof}

In a BFS-layering of a $k$-tree~$H$, each layer has treewidth at most~$k-1$ (cf. \cref{obs:bfs-layering}).
For $(w,k,t)$-sums, natural layerings have a similar property: By (N\ref{property_layering_sum:layer_w-1}), the induced subgraph of each layer of a $(w,k,t)$-sum is a subgraph of a $(w-1,k,t)$-sum. 
By induction on~$w$, we can now bound the strong odd chromatic number~$\chiso$ of subgraphs of~$(w,k,t)$-sums.

\begin{proposition}
\label{prop:sum-fine}
    Let~$k,t,w,m\geq 0$ be integers. 
    There exists a constant~$f_4(k,t,m,w)$ such that for every $(w,k,t)$-sum~$S$, every directed subgraph~$\vec{G}$ of~$S$ and every collection~$M_1, \dots, M_m$ of subsets of~$V(S)$, there is a coloring~$\Psi \colon V(S) \to [f_4(k,t,m,w)]$ of~$S$ with the following properties:
    \begin{enumerate}[(S1)]
        \item \label{enum-sum:proper} $\Psi$ is a proper coloring of the $(w,k,t)$-sum~$S$,
        \item \label{enum-sum:strong_odd} $\Psi$ restricts to a strong odd coloring on the directed subgraph~$\vec{G}$,
        \item \label{enum-sum:marks-odd} $\Psi$ is strong odd on the set~$M_j$ for every $j \in [m]$.
    \end{enumerate}
\end{proposition}

Note that \cref{prop:sum-fine} immediately implies \cref{thm:wkt-sums}. The proof is very similar to the proof of \cref{prop:treewidth-fine}.
As in \cref{sec:treewidth}, we first deduce a consequence of \cref{prop:sum-fine}, that we will use in the induction step.

\begin{lemma}
\label{lem:sum-clique-coloring}
    If \cref{prop:sum-fine} holds for fixed integers~$w,k,t \geq 0$, and every integer~$m \geq 0$, then the following statement holds as well.

    There exists a constant~$g_4(k,t,w)$ such that for every $(w,k,t)$-sum~$S$ and every subset~$B$ of cliques in~$S$, there is a coloring~$\sigma\colon B \to [g_4(k,t,w)]$ of~$B$ with the following properties:
    \begin{enumerate}[(S'1)]
        \item \label{enum-clique-coloring_clique-sums:vertex_in_odd_number_of_cliques} 
        For every vertex~$v$ of~$S$, the coloring~$\sigma$ is strong odd on the set of cliques in~$B$ containing~$v$, and
        \item \label{enum-clique-coloring_clique-sums:odd_sized_color_classes}
        every color class has odd size.
    \end{enumerate}
\end{lemma}
\begin{proof}
    The idea of the proof is the following.
    We assign to every clique~$Q_x$ of~$S$ a representative~$x \in V(S)$ that is a vertex of~$S$.
    A coloring~$\psi$ of the vertices of~$S$ then induces a coloring of the cliques~$B$ where each clique~$Q_x$ is colored with the color~$\psi(x)$ of its representative~$x$.
    Yet, a vertex might represent several cliques.
    We therefore define types of cliques.
    For every type~$a$, we let~$B(a)$ be the set of cliques in~$B$ of type~$a$ and note that each vertex in~$S$ represents at most one clique in~$B(a)$.
    The sets~$B(a)$ form a partition of~$B$. 
    For each type~$a$, we construct a vertex-coloring~$\psi_a$ of~$S$ that induces a coloring~$\sigma_a$ of the set~$B(a)$.
    The union of these colorings yields the coloring~$\sigma$.
    
    \proofsubparagraph{Types of Cliques.}
    Every clique~$Q$ of~$S$ is a clique in some~$(k,t)$-summand~$F = (H \boxtimes P) + K_t$ where~$H$ is a $k$-tree and~$P$ a path.
    For every clique, we choose one such summand (there might be several options, we choose one of them).
    The clique~$Q$ contains vertices of up to two adjacent layers~$L_d$ and~$L_{d+1}$ of the strong product~$H \boxtimes P$ and some vertices of~$K_t$.
    Recall that every layer corresponds to a copy of~$H$.
    Note that~$Q$ is contained in a clique~$Q'$ of size~$2(k+1)+t$ that contains all vertices of~$K_t$ and induces the same $(k+1)$-clique~$Q_H$ of~$H$ in two adjacent layers~$L_d$ and~$L_{d+1}$ of~$H \boxtimes P$, see \cref{fig:clique-sum_clique-covering}.
    \begin{figure}
        \centering
        \includegraphics[page=1]{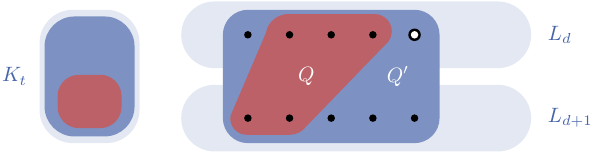}
        \caption{The clique~$Q$ (red) is contained in a clique~$Q'$ (blue) of size~$2(k+1)+t$. The vertex-ordering within each layer corresponds to the construction sequence in the~$k$-tree~$H$. The vertex representing~$Q$ is colored in white. }
        \label{fig:clique-sum_clique-covering}
    \end{figure}
    There might be several choices for~$Q'$, we fix one.
    
    The type of~$Q$ is derived from its intersection with~$L_d,L_{d+1}$ and~$K_t$.
    Let~$v_1, \dots, v_{k+1}$ be the insertion order of the vertices of~$Q_H$ in the construction sequence of~$H$.
    Now, we define~$a_1$ as the $\mathsf{0}$-$\mathsf{1}$-vector of length~$k+1$ where the $i$-th entry is~$\mathsf{1}$ if and only if~$v_i \in Q \cap L_d$.
    Similarly, we define~$a_2$ as the $\mathsf{0}$-$\mathsf{1}$-vector of length~$k+1$ where the $i$-th entry is~$\mathsf{1}$ if and only if~$v_i \in Q \cap L_{d+1}$.
    Fix an order~$u_1, \dots, u_t$ of the vertices of~$K_t$.
    We let $a_3$ be the $\mathsf{0}$-$\mathsf{1}$-vector of length~$t$ where the $i$-th entry is~$\mathsf{1}$ if and only if~$u_i \in Q \cap V(K_t)$.
    The type of~$Q$ now corresponds to the tuple~$(a_1, a_2, a_3)$.
    Note that there are at most $2^{k+2}\cdot 2^t$ different types of cliques.
    We say that the clique~$Q$ is represented by the vertex~$(v_{k+1},d)$, i.e.\ the vertex of~$Q_H$ of highest order in the $d$-th layer.
    Observe that a vertex~$(v,d) \in V(H \boxtimes P)$ represents at most one clique of type~$a$.
    We write~$Q_{(v,d)}$ for the clique represented by a vertex~$(v,d)$.
    Note in particular that a clique~$Q$ might be represented by a vertex that does not lie in~$Q$.
    
    \proofsubparagraph{Coloring Cliques of the Same Type.}
    Let~$a$ be a type and~$B(a) \coloneqq \set{Q \in B \given \text{$Q$ is of type~$a$}}$.
    As the number of types only depends on~$k$ and~$t$, we may use a different color palette for each set~$B(a)$.
    It thus suffices to find a coloring of~$B(a)$ that satisfies (S'\ref{enum-clique-coloring_clique-sums:vertex_in_odd_number_of_cliques}) and (S'\ref{enum-clique-coloring_clique-sums:odd_sized_color_classes}).
    If~$\psi$ is a vertex-coloring of~$S$, then we obtain a coloring~$\sigma$ of the cliques of type~$a$ in~$B$ by assigning every clique~$Q_{(v,d)} \in B(a)$ the color~$\psi(v,d)$ where $(v,d) \in V(H \boxtimes P)$ is the vertex representing~$Q_{(v,d)}$.

    \begin{claim}
        For every type~$a$, there exists a directed subgraph~$\vec{D}_a$ of~$S$ such that if~$\psi_a$ is a proper coloring of~$S$ that restricts to a strong odd coloring on~$\vec{D}_a$, then the induced coloring~$\sigma_a$ of~$B(a)$ satisfies~(S'\ref{enum-clique-coloring_clique-sums:vertex_in_odd_number_of_cliques}).
    \end{claim}
    \begin{claimproof}
        Let~$\vec{D}_a$ be the directed graph on the vertices of~$S$ where every vertex~$x$ is connected with a directed edge to each vertex~$y \in V(S)$ with $y \neq x$ that represents a clique in~$B(a)$ which contains~$x$, i.e.\ $E(\vec{D}) = \set{\vec{xy} \given x,y \in V(S), Q_y \in B(a), x \in Q_y}$.

        We first prove that~$\vec{D}_a$ is a subgraph of~$S$.
        Let~$x \in V(S)$ and consider a clique~$Q_y \in B(a)$ that contains~$x$ with $y \neq x$.
        If~$x \in V(K_t)$, we clearly have $\vec{xy} \in E(S)$.
        Otherwise, $x = (v,d) \in V(H \boxtimes P)$ and we write $y = (u,\ell)$. 
        As~$x \in Q_y$, it follows that $\ell \in \set{d,d-1}$. 
        Let~$Q_y'$ be the clique of size~$2(k+1)+t$ which contains~$Q_y$.
        As $x \in Q_y'$, we obtain that $v = u$ or $(v,\ell)$ and~$(u,\ell)$ are connected in~$S$. 
        It follows that $xy \in E(S)$.
        Thus, $\vec{D}_a$ is a subgraph of~$S$.

        Suppose~$\psi_a$ is proper on~$S$ and strong odd on~$\vec{D}_a$.
        It remains to show that for every vertex~$x \in V(S)$, the induced coloring~$\sigma_a$ is strong odd on the set~$B(x,a) = \set{Q \in B(a) \given x \in Q}$ of cliques containing~$x$.
        As the coloring~$\psi_a$ is strong odd on~$N^+(\vec{D}_a,x)$, the induced coloring~$\sigma_a$ is strong odd on the set~$B(x,a) - Q_x$. 
        As~$\psi_a$ is proper, the clique~$Q_x$ has a different color from every other clique in~$B(x,a)$ and it follows that~$\sigma_a$ is strong odd on~$B(x,a)$.
    \end{claimproof}

    \begin{claim}
        For every type~$a$, there exists a set~$V_{B(a)} \subseteq V(S)$ such that if~$\psi_a$ is a coloring of~$S$ that is strong odd on~$V_{B(a)}$, then the induced coloring~$\sigma_a$ of~$B(a)$ satisfies (S'\ref{enum-clique-coloring_clique-sums:odd_sized_color_classes}).
    \end{claim}
    \begin{claimproof}
        Let~$V(B(a)) = \set{x \in V(H \boxtimes P) \given Q_x \in B(a)}$ be the set of vertices which represent cliques in~$B(a)$. 
        Recall that the color classes of~$\sigma_a$ correspond to the color classes of the vertices~$V_{B(a)}$, i.e., for every color~$c$ of~$\sigma_a$, we have $\abs{B(a) \cap \sigma_{a}^{-1}(c)} = \abs{V_{B(a)} \cap \psi_{a}^{-1}(c)}$. 
        As~$\psi_a$ is strong odd on~$V_{B(a)}$, each color class of~$\sigma_a$ has odd size.
    \end{claimproof}

    If \cref{prop:sum-fine} holds for~$w,k,t$ and~$m = 1$, then there exists for every type~$a$ a proper coloring~$\psi_a$ of the $(w,k,t)$-sum~$S$ with at most~$f_4(k,t,1,w)$ colors that restricts to a strong odd coloring on the directed graph~$\vec{D}_a$ and is strong odd on the set~$V_{B(a)}$.
    By the two claims above, the union~$\sigma$ of the colorings~$\sigma_a$ fulfills~(S'\ref{enum-clique-coloring_clique-sums:vertex_in_odd_number_of_cliques}) and~(S'\ref{enum-clique-coloring_clique-sums:odd_sized_color_classes}). 
    As there are at most~$2^{k+2+t}$ different types, $\sigma$ uses at most~$g_4(k,t,w) = 2^{k+2+t} \cdot f_4(k,t,1,w)$ colors.
\end{proof}

\begin{proof}[Proof of \cref{prop:sum-fine}]
We prove the claim by induction on the size~$w$ of the largest clique on which we attach a summand of~$S$ in its construction sequence. 

If~$w=0$, the $(w,k,t)$-sum~$S$ is the disjoint union of $(k,t)$-summands~$F_1, \dots, F_{\ell}$.
By \cref{lem:summand-fine}, there exists for every summand~$F_i$ a proper coloring~$\Psi_i$ with at most~$f_3(k,t,m)$ colors that restricts to a strong odd coloring on the directed graph~$F_i \cap \vec{G}$ and is strong odd on the sets~$M_j \cap V(F_i)$ for~$j \in [m]$.
Yet, the union of the colorings~$\Psi_i$ might not be strong odd on the sets~$M_j$.

We assign to every summand~$F_i$ a type~$A_i$ based on the coloring~$\Psi_i$.
The type~$A_i$ is a $\mathsf{0}$-$\mathsf{1}$-matrix with $m$~rows and~$f_3(k,t,m)$ columns defined as follows:
\begin{itemize}
    \item For~$c \in [f_3(k,t,m)]$, column~$c$ corresponds to color~$c$ in the coloring~$\Psi_i$ of~$F_i$.
    \item For~$j \in [m]$, row~$j$ corresponds to the set~$M_j \cap V(F_i)$. The matrix~$A_i$ has a~$1$ in row~$j$ and column~$c$ if and only if at least one vertex in~$M_j \cap V(F_i)$ has color~$c$ in~$\Psi_i$.
\end{itemize}
Note that there are only~$2^{m \cdot f_3(k,t,m)}$ different types.

For every type~$A$, consider the $0$-tree~$T_a$ whose vertices correspond to the summands~$F_i$ of type~$A$.
By \cref{prop:treewidth-fine}, there exists a coloring~$\sigma_A$ of~$T_a$ with at most~$f_1(0,0,m)$ colors that is strong odd on the sets~$M_j' = \set{F_i \given i \in [\ell], \text{$F_i$ is of type~$A$}, M_j \cap V(F_i) \neq \varnothing}$ for all~$j \in [m]$.

We now define the coloring~$\Psi$ of~$S$.
For a vertex~$v \in S$ that belongs to the summand~$F_i$ of type~$A$, we set~$\Psi(v) = (\Psi_i(v),A,\sigma_A(F_i))$.
Note that~$\Psi$ uses at most 
\[f_4(k,t,m,0) \coloneqq f_3(k,t,m) \cdot 2^{m \cdot f_3(k,t,m)}\cdot f_1(0,0,m)\]
colors.
The same arguments as in the proof of \cref{prop:treewidth-fine} show that~$\Psi$ fulfills~(S\ref{enum-sum:proper})-(S\ref{enum-sum:marks-odd}).

\smallskip

Now let~$w \geq 1$.
As in \cref{prop:treewidth-fine}, we first construct a vertex-coloring~$\Phi$ of~$S$ that is the product of four colorings~$\varphi_1, \varphi_2, \varphi_3,\varphi_4$ and satisfies~(S\ref{enum-sum:proper})-(S\ref{enum-sum:strong_odd}).
A final modification then yields the desired coloring~$\Psi$.

By \cref{lem:sum_natural_layering}, there exists a natural layering~$L_1, \dots, L_{\ell}$ of the graph~$S$.
The coloring~$\Phi$ independently colors each layer~$L_d$.
First suppose that~$d \geq 2$.
As the $(d-1)$-th layer induces a subgraph of a $(w-1,k,t)$-sum, there exists by induction a proper coloring~$\chi_{d-1}$ of~$L_{d-1}$ with at most~$f_4(k,t,0,w-1)$ colors.

\proofsubparagraph{Construction of~$\bm{\varphi_1}$.}
Let~$C$ be a component of~$L_d$.
The neighbors of $C$ in~$L_{d-1}$ form a clique~$Q$ by (N\ref{property_layering_sum:parent-clique}).
We call the vertices of~$C$ \emph{children} of~$Q$ and say that~$Q$ is their \emph{parent-clique}.
A clique~$Q$ in~$L_{d-1}$ may have children in several components of~$L_d$, but every vertex of~$L_d$ has exactly one parent-clique.
For every clique~$Q$ in~$L_{d-1}$, let~$S_Q$ be the subgraph of~$S$ induced by its children.
We define a coloring~$\varphi_Q$ of~$S_Q$.
For a vertex~$v \in L_d$, the color~$\varphi_1(v)$ corresponds to its color in~$\varphi_Q$, i.e., $\varphi_1(v) = \varphi_Q(v)$.

The coloring~$\varphi_Q$ is obtained as follows.
By (N\ref{property_layering_sum:parent-clique}), we may assume that the clique~$Q \subseteq L_{d-1}$ has size~$q \leq w$.
Let~$v_1, \dots, v_q$ denote the vertices of~$Q$.
Each vertex~$v_i$ has a different color~$c_i$ in~$\chi_{d-1}$.
For every~$i \in [q]$, let~$N_{c_i,Q} = N^+(\vec{G},v_i) \cap V(S_Q)$ denote the out-neighbors of~$v_i$ in~$\vec{G}$ that are in~$S_Q$.
For~$i \in [m]$ let~$M_{i,Q} = M_i \cap V(S_Q)$ denote the restrictions of the sets~$M_i$ to~$S_Q$ and let~$\vec{G}_Q$ be the restriction of~$\vec{G}$ to~$S_Q$.

The graph~$S[L_d]$ is a subgraph of a $(w-1,k,t)$-sum~$S'$ by (N\ref{property_layering_sum:layer_w-1}).
Calling induction on the graph~$S'$ with the directed graph~$\vec{G}_Q$ and the sets~$M_{i,Q}$ with $i \in [m]$ and~$N_{c_i,Q}$ with $i \in [q]$ yields the coloring~$\varphi_Q$ with at most $f_4(k,t,m+w,w-1)$~colors.

\proofsubparagraph{Construction of~$\bm{\varphi_2}$.}
We assign to every clique~$Q \in L_{d-1}$ of size at most~$w$ a \emph{type}. 
A type is a $\mathsf{0}$-$\mathsf{1}$-matrix~$A_Q$ with $m+f_4(k,t,0,w-1)$~rows and $f_4(k,t,m+w,w-1)$~columns defined as follows:
\begin{itemize}
    \item For $c \in [f_4(k,t,m+w,w-1)]$, column~$c$ corresponds to color~$c$ in the coloring~$\varphi_Q$ of~$S_Q$.
    \item For $j \in [m]$, row~$j$ corresponds to the set~$M_{j,Q} = M_j \cap V(S_Q)$. 
    The matrix~$A_Q$ has a~$\mathsf{1}$ in row~$j$ and column~$c$ if and only if at least one vertex in~$M_{j,Q}$ in~$S_Q$ has color~$c$ in~$\varphi_Q$.
    \item For $h \in [f_4(k,t,0,w-1)]$, $h$ corresponds to a color of~$\chi_{d-1}$.
    If there is no vertex~$v \in Q$ with~$\chi_{d-1}(v) = h$, row~$m+h$ contains only zeros.
    Otherwise, there is exactly one vertex~$v$ in~$Q$ with~$\chi_{d-1}(v) = h$.
    Here, $A_Q$ has a~$\mathsf{1}$ in row~$m+h$ and column~$c$ if and only if at least one neighbor~$u \in N_{h,Q}$ of~$v$ in~$\vec{G}_Q$ has color~$c$ in~$\varphi_Q$.
\end{itemize}
For a vertex~$v \in L_d$ with parent-clique~$Q$, we set~$\varphi_2(v) = A_Q$.
Note that~$\varphi_2$ admits at most~$(m+f_4(k,t,0,w-1)) \cdot f_4(k,t,m+w,w-1)$ different colors.

\proofsubparagraph{Construction of~$\bm{\varphi_3}$.}
For every type~$A$, let~$\sigma_A$ be the coloring of the cliques of~$L_{d-1}$ of type~$A$ with $g_4(k,t,w-1)$~colors we obtain by an application of \cref{lem:sum-clique-coloring}.
For a vertex~$v \in L_d$ with parent-clique~$Q$ of type~$A$, we set~$\varphi_3(v) = \sigma_{A}(Q)$.

\proofsubparagraph{Definition of~$\bm{\Phi}$.}
We now define the vertex-coloring~$\Phi$ for each layer~$L_d$.

Let~$d \geq 2$ and let~$v \in L_d$ be a vertex with parent-clique~$Q$ of type~$A$.
We define~$\Phi(v) = (\varphi_1(v),\varphi_2(v),\varphi_3(v),\varphi_4(v))$ where
\begin{itemize}
    \item $\varphi_1(v) = \varphi_Q(v)$ is the color of~$v$ under the partial coloring~$\varphi_Q$ of~$G[L_d]$,
    \item $\varphi_2(v) = A$ is the type of~$Q$,
    \item $\varphi_3(v) = \sigma_A(Q)$ is the color of the parent-clique~$Q$ in the coloring~$\sigma_A$ of all cliques of the $(d-1)$-th layer of type~$A$, and
    \item $\varphi_4(v) = d \bmod 3$ is the index of the layer~$L_d$ modulo~$3$.
\end{itemize}

Now let $d=1$. The vertices of~$L_1$ induce a $(0,k,t)$-sum.
Let~$\varphi_1$ be a proper coloring of~$L_1$ that is strong odd on the restriction~$\vec{G}[L_1]$ of~$\vec{G}$ and the restrictions~$M_i \cap L_1$ of the sets~$M_i$ to the first layer.
By induction, there exists such a coloring with at most~$f_4(k,t,m,0)$ colors.
For a vertex~$v \in L_1$, we set $\Phi(v) = (\varphi_1(v),-1,-1,1 \bmod 3)$.

As the coloring~$\varphi_1$ uses at most~$f_4(k,t,m+w,w-1)$ colors, the coloring~$\varphi_2$ at most~$(m+f_4(k,t,0,w-1)) \cdot f_4(k,t,m+w,w-1) + 1$ colors, $\varphi_3$ at most~$g_4(k,t,q-1)+1$ colors and~$\varphi_4$ at most three colors, the number of colors of~$\Phi$ is bounded in terms of~$m,k,t$ and~$w$.

The same argument as in the proof of~\cref{prop:treewidth-fine} shows that~$\Phi$ satisfies (S\ref{enum-sum:proper})-(S\ref{enum-sum:strong_odd}).
Using the same technique as in the proof of~\cref{prop:treewidth-fine}, we obtain a coloring~$\Psi$ of~$S$ that uses at most twice as many colors as~$\Phi$ and satisfies (S\ref{enum-sum:proper})-(S\ref{enum-sum:marks-odd}). 
\end{proof}

\section{Application to Facially Odd Colorings}\label{sec:app}

In~\cite{CJ09} the authors introduce \emph{proper facially odd colorings} as proper vertex-colorings of a $2$-connected plane graph~$G$ such that every face contains every color an odd number of times or not at all. 
We denote by $\chi_{\mathrm{pfo}}(G)$ the corresponding parameter. 
In~\cite[Conjecture 7.1]{CJ09} the authors wonder whether $\chi_{\mathrm{pfo}}(G)$ is bounded by a global constant on the class of $2$-connected plane graphs. 
This was shown by~\cite{CJV11} and the current best bound of $97$ is due to~\cite{KRSS14}. 
Special classes of plane graphs have been studied in~\cite{CJK11,WFW12}. 
Clearly, the notion of proper facially odd colorings as well as the parameter $\chi_{\mathrm{pfo}}(G)$ extend immediately to $G$ with a closed $2$-cell-embedding on a closed surface $S$, i.e.\ every face is bounced by a cycle. We can prove a strengthening of~\cite[Conjecture 7.1]{CJ09}:
\begin{corollary}\label{parity}
For any surface $S$ there exists a constant $c$ such that every closed $2$-cell-embedded graph $G$ has $\chi_{\mathrm{pfo}}(G)\leq c$.
\end{corollary}\label{cor:faces}
\begin{proof}
Take a closed $2$-cell-embedded graph $G$ in $S$ and add one vertex $v_f$ into each face $f$ and connect $v_f$ to the vertices of $f$. The resulting simple graph $G'$ clearly also embeds into $S$. Since the class of graphs that embed into $S$ is proper minor-closed \cref{thm:minor-closed} yields $\chi_{so}(G')\leq c$. The restriction of this coloring to $G$ is a proper vertex-coloring such that every face contains every color an odd number of times or not at all. Hence, $\chi_{\mathrm{pfo}}(G)\leq c$.
\end{proof}

Interestingly, the fact that $\chi_{\mathrm{pfo}}$ is bounded for planar graphs~\cite{CJV11} is a crucial ingredient for the upper bound on the strong odd chromatic number of planar graphs in~\cite{caro2024strong}.

\section{Going further}\label{sec:further}

\paragraph*{Beyond proper}
In this section, we drop the condition of the coloring being a proper coloring and consider \emph{improper strong odd colorings} of a graph $G$ and the \emph{improper strong odd chromatic number} $\chi_{\mathrm{iso}}(G)$.
We show:

\begin{proposition}
    There are $n$-vertex graphs that need $\Omega (n)$ colors in any improper strong odd coloring.
\end{proposition}
\begin{proof}
    Let $G=K_n$ where $n$ is odd. Our objective is to prove that $\chi_{\mathrm{iso}}(G)=n$.

    Assume, for contradiction, the existence of an improper strong odd coloring where a color class $C$ contains at least 2 vertices.

    If $|C|$ is odd, then each vertex in $C$ sees an even number of vertices of its own color among its neighbors, which means it is not a strong odd coloring. Thus, $|C|$ must be even. However, since $n$ is odd, there exists a vertex $v \in V(G)\setminus C$. Hence, $v$ sees  each vertex in the class $C$, implying that $v$ sees a color an even number of times among its neighbors. 
\end{proof}

\begin{proposition}
    There is no function $f$ such that $\chiso(G)\leq f(\chi_{\mathrm{iso}}(G),\chi(G))$.
\end{proposition}
\begin{proof}
    Consider a graph $G$ with vertex set $V(G) = V_1 \cupdot V_2 \cupdot V_3 \cupdot V_4$ where $V_1$ and $V_4$ have $n$ vertices each and $V_2$ and $V_3$ have $\binom{n}{2}$ vertices each.
    Vertices in $V_2$ are connected to the different pairs of vertices in $V_1$. 
    Vertices in $V_3$ are true twins (i.e., have the same closed neighborhood) of each vertex in $V_2$. 
    Finally, vertices in $V_4$ are pendants of each vertex in $V_1$.
    
    Graph $G$ has chromatic number $3$ (e.g., with color classes $V_1$, $V_2$, and $V_3\cup V_4$). 
    The improper odd chromatic number~$\chi_{\mathrm{iso}}(G)$ is $1$, since all vertices of~$G$ have odd degree. 
    However, the strong odd chromatic number~$\chiso(G)$ is at least~$n$ since no color can be repeated in $V_1$. Assume for a contradiction that a color $c$ was repeated in vertices $u,v \in V_1$. 
    Neither of the shared neighbors of $u$ and $v$ in $V_2$ and $V_3$ can have color $c$. But then they see color $c$ exactly twice.
\end{proof}

\paragraph*{Beyond odd}

Let $S=\{s\in \mathbb{N}\mid s\bmod 2=1\}\cup\{0\}$ be the set of odd numbers and $0$, then a $t$-coloring $\varphi$ of~$G$ is strong odd if $\abs{\varphi(i)^{-1}\cap N(v)}\in S$ for all $v\in V(G)$ and $i\in[t]$. We have the feeling that our proofs could be extended to the more general setting where for some $n\in\mathbb{N}$ and some generating set $C$ of $\mathbb{Z}_n$ we define $S=\{s\in \mathbb{N}\mid s\bmod n\in C\}\cup\{0\}$. This also includes the setting $C=\{1\}$ with general $n$  considered in~\cite{Bom13}. Indeed, a central property of $S$ used in our proofs is that it is a \emph{basis of bounded order}, i.e., there exists $c\in\mathbb{N}$ such that for every~$n\in\mathbb{N}$ there are $(s_i)_{1\leq i\leq n}$ in $S$ such that $\sum_{i=1}^cs_i=n$.

Indeed, for example one can see along the proof for bounding the strong odd chromatic number of trees~\cite[Proposition 2.1]{caro2024strong} that if $S$ is a basis of bounded order with constant $c$, then any tree $G$ can be $(c+2)$-colored such that $|\varphi(i)^{-1}\cap N(v)|\in S$ for all $v\in V(G)$ and $i\in[c+2]$.

A closer inspection however yields that our proofs furthermore require, that  $S$ is \emph{multiplication-closed}, i.e., if $s,t\in S$ then $st\in S$.
An interesting example of a multiplication closed basis comes from the Hilbert–Waring theorem~\cite{Hil09}: for any fixed $n\in\mathbb{N}$ there exists a $c\in\mathbb{N}$ such that all $n$th powers of natural numbers form a multiplication-closed basis of bounded order $c$. 

\begin{question}
    Does there exist a constant $c$, such that every planar graph can be properly vertex-$c$-colored such that every color in a neighborhood appears a square number of times? 
\end{question}

Note that similar extensions of the ordinary odd chromatic number have been considered in~\cite{Liu24}.

\paragraph*{Beyond minor-closed via hypergraphs}

Given a hypergraph $H=(V,E)$, a vertex coloring is strong odd if every hyperedge contains every color either an odd number of times or not at all. The parity colorings considered in~\cite{CJ09} are an instance of this. Further, note that this is the strong variant of the odd notion considered in~\cite{CKP13}. We wonder:

\begin{question}\label{quest:hyper}
    Do hypergraphs of bounded degree have bounded strong odd chromatic number?
\end{question}

Note that by hypergraph duality~\cref{quest:hyper} can also be formulated as a statement about \emph{strong odd edge colorings} of hypergraphs of bounded uniformity, where now edges must be colored such that every vertex sees every edge-color an odd number of times or never. The corresponding question for $2$-uniform hypergraphs, i.e., graphs, has been answered in the positive~\cite{LPS15}, namely $6$ colors always suffice.  However, already it is open whether there exists a constant $c$ such that every $3$-uniform hypergraph has a strong odd edge coloring with $c$ colors.

An answer to~\cref{quest:hyper} would be interesting because of the following. Here the \emph{star-chromatic number} $\chi^*(G)$ is the smallest number of colors for a proper vertex coloring such that the union of any two color-classes induces a forest of stars.

\begin{proposition}
    If for any maximum degree $\Delta$ there exists a constant $c$, so that $\chi_{so}(H)\leq c$ for every hypergraph $H$ of maximum degree $\Delta$, then the strong odd chromatic number of any graph $G$ is bounded in terms of its star-chromatic number $\chi^*(G)$.
\end{proposition}
\begin{proof}
    Let $G=(V,E)$ be a graph and $\varphi$ a star-coloring with $k=\chi^*(G)$ colors. Define $H$ on $V$ such that $X\subseteq V$ is an edge if $X$ are the extremities of a $2$-colored star with respect to $\varphi$. In case a star has at most $2$ vertices, choose any of its vertices as its extremity. The maximum degree $\Delta$ of $H$ is less than $k$, so let $\psi$ be a strong odd coloring of $H$ with $c_{k}$ colors. As a strong odd coloring of $G$ we propose $f(v)=(\varphi(v),\psi(v))$. The first entry assures the coloring to be proper. Let $w\in N(V)$. If the star in $\varphi^{-1}(v)\cup \varphi^{-1}(w)$ that contains $v,w$ has $w$ as its center, then $w$ is unique in $N(v)$ with respect to $\varphi$, hence $f(w)$ also appears exactly once in $N(v)$. If otherwise that star has $v$ as its center, then all the neighbors of $v$ that have the same color with respect to $\varphi$ as $w$ form an edge of $H$. Hence, since $\psi$ is a strong odd coloring of $H$, they are form parts of odd size each with resect to their second coordinate. Thus, $f$ is a strong odd proper coloring with $c_{k}k$ colors.
\end{proof}

Hence, since by~\cite{Nešetřil2003} every proper minor-closed classes has bounded star-chromatic number, a positive answer to~\cref{quest:hyper} would imply~\cref{thm:minor-closed}. Even stronger, as every class of bounded expansion has bounded star-chromatic number (see e.g., \cite[p.\,188]{jiang2023chi}), a positive answer to~\cref{quest:hyper} would give a positive answer to:

\begin{question}
    \label{quest:bounded_expansion}
    Is the strong odd chromatic number bounded on graph classes of bounded expansion?
\end{question}

Note that in~\cite{Hic23,Liu24} it has been shown that the proper conflict free chromatic number and hence in particular the odd chromatic number is bounded on graph classes of bounded expansion.

\subsection*{Acknowledgments}
Part of this work was conducted at the 12th Annual Workshop on Geometry and Graphs held at the Bellairs Research Institute in February 2025. We are grateful to the organizers and participants for providing an excellent research environment.

\bibliography{refs}

\end{document}